\documentclass[11pt,reqno]{amsart}
\usepackage{amssymb}
\usepackage{enumerate}
\usepackage{amsmath,amssymb,amsfonts,amsthm,
	latexsym, amscd, epsfig, epic,color}
\usepackage{epstopdf}
\usepackage{mathtools}
\usepackage{youngtab}
\usepackage{extarrows}
\usepackage{enumitem}
\usepackage{enumerate}
\usepackage{mathrsfs}
\usepackage{eucal}
\usepackage{eufrak}
\usepackage{xspace}
\usepackage{a4wide}
\usepackage{colordvi}
\usepackage{comment}
\usepackage{hyperref,eucal}
\usepackage{tabularx}
\usepackage{tikz}\usetikzlibrary{matrix,arrows}
\usepackage{ytableau} \usepackage{chemfig}
\usepackage{float}
\usepackage{youngtab}
\usepackage{mathrsfs}
\usepackage{caption}
\usepackage{epsfig}
\usepackage{stfloats}
\usepackage{booktabs}
\usepackage{changepage}
\usepackage{graphicx}
\usepackage{xcolor}
\usepackage{pict2e}
\theoremstyle{plain}

\theoremstyle{plain}

\newtheorem{theorem}{Theorem}[section]
\newtheorem{corollary}{Corollary}[section]

\newtheorem{lemma}{Lemma}[section]
\newtheorem{proposition}{Proposition}[section]
\newtheorem{conjecture}{Conjecture}[section]
\newtheorem{definition}{Definition}[section]

\theoremstyle{example}
\newtheorem{example}{Example}[section]

\theoremstyle{openproblem}

\def\E{\CMcal{E}}

\def\G{\CMcal{G}}

\def\A{\mathcal{A}}

\def\S{\CMcal{S}}
\def\H{\CMcal{H}}
\def\L{\CMcal{L}}
\def\GL{\mathrm{GL}}
\def\inv{\mathrm{inv}}

\def\r{\mathbf{r}}
\def\bnu{\boldsymbol{\nu}}
\def\maj{\mathrm{maj}}

\def\dg{\mathrm{dg}}

\def\arm{\mathrm{arm}}
\def\leg{\mathrm{leg}}
\def\asc{\mathrm{asc}}

\def\col{\mathrm{col}}
\def\row{\mathrm{row}}
\def\SSYT{\mathrm{SSYT}}
\def\Exp{\mathrm{Exp}}
\def\bT{\boldsymbol{T}}

\def\SYT{\mathrm{SYT}}
\newcommand{\NN}{\mathbb{N}}

\newcommand{\coinv}{\operatorname{coinv}}

\makeatletter
\@namedef{subjclassname@2020}{\textup{2020} Mathematics Subject Classification}
\makeatother

\theoremstyle{remark}

\newtheorem{remark}{Remark}
\numberwithin{equation}{section}
\numberwithin{figure}{section}

\title[Positivity for a $t$-analogue of refined Littlewood-Richardson rules]{Positivity for a $t$-analogue of refined Littlewood-Richardson rules}

\author{Jiaming Du}
\address{School of Mathematical Sciences, Xiamen University, Xiamen 361005, China}
\email{jiamingduu@outlook.com}

\author{Emma Yu Jin}
\address{School of Mathematical Sciences, Xiamen University, Xiamen 361005, China}
\email{yjin@xmu.edu.cn}

\author{Xiaowei Lin}
\address{School of Mathematical Sciences, Xiamen University, Xiamen 361005, China}
\email{linxiaoweiqing@126.com}


\date{\today}

\begin{document}
	
	\begin{abstract}
		In 2010, Haglund, Luoto, Mason and van Willigenburg proved that the product of a Schur function $s_{\lambda}$ and a Demazure atom $\A_{\gamma}$ expands positively in terms of Demazure atoms, and interpreted the coefficients as the number of Littlewood-Richardson skew skyline tableaux of a given shape and a filling type. Extending this approach, they further established the positive expansion for the product of a Schur function $s_{\lambda}$ and a quasisymmetric Schur function $\S_{\alpha}$, or of a Schur function $s_{\lambda}$ and a Demazure character $\kappa_{\gamma}$. In particular, the latter provides a refinement of the classical Littlewood-Richardson rule.
		
		We present a natural $t$-analogue of these positive expansions by deriving the expansion for a product of a Hall-Littlewood polynomial $P_{\lambda}(t)$ and a $t$-Demazure atom $\E_{\gamma}(t)$ in terms of $t$-Demazure atoms. The corresponding coefficients are $\mathbb{N}\,[t,1-t]$-positive; that is, they are polynomials with nonnegative integer coefficients in $t$ and $1-t$, and they are interpreted in terms of standard Young tableaux. We take the approach of connecting the multiplication coefficients with the LLT polynomials, which differs from the combinatorial insertion arguments of Haglund, Luoto, Mason and van Willigenburg.
		
	\end{abstract}
	\maketitle
	
	\section{Introduction and main results}
	
	The Littlewood-Richardson (LR) rule expresses the product of two Schur functions as
	\begin{align}\label{E:LR1}
		s_{\mu}s_{\nu}&=\sum_{\lambda}c_{\mu\nu}^{\lambda}s_{\lambda},
	\end{align}
	where the coefficients $c_{\mu\nu}^{\lambda}$ are nonnegative integers that arise naturally in representation theory, combinatorics and algebraic geometry, etc; see, for example, \cite{Fulton:97,FH:91,Sagan:00,Stanley,Stembridge:02}. Over the years, several combinatorial models have been developed to interpret the coefficients, including Young tableaux, honeycomb configurations and puzzles \cite{Fulton:97,KnuTao:99,KTW:04}. 
	Refinements of the Schur basis lead to corresponding questions about the positivity of multiplication in nonsymmetric and quasisymmetric bases.

	Haglund, Luoto, Mason and van Willigenburg extended the LR rule (\ref{E:LR1}) to the product of a Schur function $s_{\lambda}$ and a Demazure character $\kappa_{\gamma}$ (a non-symmetric generalization of Schur functions), proving that this product expands positively in terms of Demazure characters. More precisely,
	\begin{theorem}\cite[Theorem 6.1 and Section 6.1]{HLMvW10}
		For any weak composition $\gamma=(\gamma_1,\ldots,\gamma_n)\in \mathbb{N}^{n}$ and a partition $\lambda$  with $\ell(\lambda)\leq n$,
		let $\gamma^r=(\gamma_n,\ldots,\gamma_1)$ denote the reverse of $\gamma$. Then
		\begin{align}\label{E:sk1}
			\kappa_{\gamma}s_{\lambda}=\sum_{\eta}d_{\gamma\lambda}^{\eta}\kappa_{\eta},
		\end{align}
		with the coefficient $d_{\gamma\lambda}^{\eta}$ denoting the number of Littlewood-Richardson skew keys of shape $\eta^r/\gamma^r$ with filling type $\lambda^r$. If $\gamma^r=(\gamma_n\ge \cdots\ge \gamma_1\ge 0)=\mu$ is a partition, then $\kappa_{\gamma}=s_{\mu}$ and (\ref{E:sk1}) reduces to (\ref{E:LR1}).
	\end{theorem}
	Since every Schur function is a special Demazure character, specifically $s_{\mu}=\kappa_{\mu^r}$, Equation (\ref{E:sk1}) yields (\ref{E:LR1}) upon observing that
	\begin{align*}
		\sum_{\eta}d_{\mu^r\lambda}^{\eta}\kappa_{\eta}=\sum_{\nu}c_{\mu\lambda}^{\nu}s_{\nu};
	\end{align*}
	for further details, see \cite[Section 6.1]{HLMvW10}. In the proof of Equation (\ref{E:sk1}), Haglund, Luoto, Mason and van Willigenburg developed combinatorial rules for the product of a Schur function $s_{\lambda}$ and a Demazure atom $\A_{\gamma}$, demonstrating its positive expansion in terms of Demazure atoms:
	\begin{theorem}\label{T:Hthm4.1}\cite[Theorem 4.1]{HLMvW10}
		For any weak composition $\gamma\in \mathbb{N}^{n}$ and a partition $\lambda$  with $\ell(\lambda)\leq n$, we have
		\begin{align}\label{E:as1}
			\A_{\gamma}s_{\lambda}=\sum_{\eta}a_{\gamma\lambda}^{\eta}\A_{\eta},
		\end{align}
		where the coefficient $a_{\gamma\lambda}^{\eta}$ counts the Littlewood-Richardson skew skyline tableaux of shape $\eta/\gamma$ with filling type $\lambda^r$.
	\end{theorem}
	
	This paper is motivated by the observation that the $t$-Demazure atom (also called the non-symmetric Hall-Littlewood polynomial) $\E_{\gamma}(t)$ and the Hall-Littlewood polynomial $P_{\lambda}(t)$ are natural $t$-analogues of the Demazure atom $\A_{\gamma}=\E_{\gamma}(0)$ and the Schur function $s_{\lambda}=P_{\lambda}(0)$, respectively.
	
	The objective of this paper is to establish
	the $\mathbb{N}[t,1-t]$-{\em positive} expansion for the product $\E_{\gamma}(t)P_{\lambda}(t)$ in terms of $t$-Demazure atoms $\E_{\eta}(t)$, meaning that every coefficient in this expansion is a non-negative integer linear combination of monomials $t^a(1-t)^b$,  with $a,b\in\mathbb{N}$. This gives a natural $t$-analogue of (\ref{E:as1}). More formally,
	\begin{theorem}\label{T:main}
		For a weak composition $\gamma\in \mathbb{N}^n$ and a partition $\lambda$  with $\ell(\lambda)\leq n$, in the expansion
		\begin{align}\label{E:defa}
			\E_{\gamma}(t)P_{\lambda}(t)=\sum_{\eta\ge \gamma}a_{\gamma\lambda}^{\eta}(t)\E_{\eta}(t),
		\end{align}
		the coefficient $a_{\gamma\lambda}^{\eta}(t)\in\mathbb{N}[t,1-t]$ is given by
		\begin{align}\label{E:productEP}
			a_{\gamma\lambda}^{\eta}(t)=\sum_{T\in\SYT(\lambda';\eta,\gamma)}
			t^{\inv^*(T)+h-m}(1-t)^{\beta^*(T)}
			\prod_{j\in \L(T)- \alpha^*(T)}[\arm^*_{j}(T)+1]_t,
		\end{align}
		where $h=h(\gamma,\eta)$ and $m=m(\gamma,\eta)$ are given in Definition \ref{def:inv-attack} and the other statistics in Definition \ref{Def:etagamma}.
		In particular, when $t=0$, we obtain (\ref{E:as1}).
	\end{theorem}
	We present two consequences of Theorem \ref{T:main}, after introducing quasisymmetric Hall-Littlewood polynomials and $t$-Demazure characters.
	Let $\CMcal{L}_{\alpha}(X_n;t)$ be the {\em \color{blue} quasisymmetric Hall-Littlewood polynomial} defined by
	\begin{align}\label{E:defL}
		\CMcal{L}_{\alpha}(X_n;t)=\CMcal{L}_{\alpha}(t)=\sum_{\gamma^+=\alpha}\E_{\gamma}(t),
	\end{align}
	where $\alpha$ is a composition of length at most $n$; see \cite[Equation (7.9)]{HLMvW11}. Then the polynomial $\CMcal{L}_{\alpha}(t)$ interpolates the quasisymmetric Schur function $\CMcal{S}_{\alpha}=\CMcal{L}_{\alpha}(0)$ and the monomial quasisymmetric function $M_{\alpha}=\CMcal{L}_{\alpha}(1)$.

	Moreover, let $T_i$ be the Demazure-Lusztig operator defined in (\ref{Def:T}), and let $\Pi_i=T_i+1-t$ for $1\le i<n$. Then the operators $\Pi_i$ also satisfy the braid relations in subsection \ref{ss:2.1}, so that $\Pi_w=\Pi_{i_1}\Pi_{i_2}\cdots \Pi_{i_l}$ for any reduced expression $w=s_{i_1}s_{i_2}\cdots s_{i_l}$ is independent of the chosen reduced word. The {\em \color{blue} $t$-Demazure character}, also called the {\em \color{blue} $t$-key polynomial}, indexed by a weak composition $\gamma\in \mathbb{N}^n$, is defined by
	\begin{align}\label{eq:keydef}
		K_{\gamma}(t)=K_{\gamma}(X_n;t)=\Pi_{w_\gamma}x^{\lambda(\gamma)},
	\end{align}
	where $\lambda(\gamma)=(\lambda_1\ge \cdots\ge\lambda_n\ge 0)$ is the unique partition determined by all parts of $\gamma$, and $w_{\gamma}$ denotes any permutation such that $w_{\gamma}\lambda(\gamma)=\gamma$. Note that (\ref{eq:keydef}) does not depend on the chosen permutation $w_{\gamma}$.
	
	\begin{corollary}\label{cor1}For a composition $\alpha$ and a partition $\lambda$, both of length at most $n$, in the expansion $$\CMcal{L}_{\alpha}(t)P_{\lambda}(t)=\sum_{\beta}b_{\alpha\lambda}^{\beta}(t)\CMcal{L}_{\beta}(t),$$ the coefficient $b_{\alpha\lambda}^{\beta}(t)\in\mathbb{N}[t,1-t]$ is given by
		\begin{align}\label{E:b2a}
			b_{\alpha\lambda}^{\beta}(t)
			=\sum_{\substack{\gamma^+=\alpha,\\ \gamma\le\eta}}a_{\gamma\lambda}^{\eta}(t),
		\end{align}
		for each weak composition $\eta\in\mathbb{N}^n$
		such that $\eta^+=\beta$.
		In particular, $b_{\alpha\lambda}^{\beta}(0)=\langle \CMcal{S}_{\beta}\rangle \CMcal{S}_{\alpha}s_{\lambda}$ equals the number of Littlewood-Richardson skew composition tableaux of shape $\beta/\alpha$ and content $\lambda^r$, as proved in \cite[Theorem 5.1, Proposition~3.4 and Remark~3.5]{HLMvW10}.
	\end{corollary}

	\begin{corollary}\label{cor21}
		For a weak composition $\gamma\in\mathbb N^n$ and a partition $\lambda$ with $\ell(\lambda)\leq n$, in the expansion
		\begin{align*}
			K_\gamma(t)P_\lambda(t)
			=\sum_{\eta\ge \gamma}d_{\gamma\lambda}^{\eta}(t)K_\eta(t),
		\end{align*}
		the coefficient $d_{\gamma\lambda}^{\eta}(t)\in\mathbb{N}[t,1-t]$ is given by
		\begin{align}\label{eq:keycoeff}
			d_{\gamma\lambda}^{\eta}(t)&= t^{\inv(\eta)-\inv(\gamma)}\,a_{\gamma^r \lambda}^{\eta^r}(t),
		\end{align}
		where $\inv(\gamma)=\#\{1\le i<j\le n\,\vert\,\gamma_i>\gamma_j\}$ denotes the inversion number of $\gamma=(\gamma_1,\ldots,\gamma_n)$. In particular, when $t=0$, we obtain (\ref{E:sk1}).
	\end{corollary}
	
	The rest of the paper is organized as follows. In Section \ref{S:2}, we present preliminaries on the combinatorics of symmetric and non-symmetric polynomials. In Sections \ref{S:3} to \ref{S:5}, we prove Theorem \ref{T:main} and  Corollaries \ref{cor1} and \ref{cor21}. Finally, we conclude with some remarks in Section \ref{S:6}.

		{\color{blue} Declaration of AI use:} The research questions and main ideas originated with the authors. ChatGPT (OpenAI, GPT-6 Astra) assisted in executing our ideas and proofreading the manuscript. The authors have independently developed the proofs in Section \ref{S:4} and of Corollary \ref{cor1}, checked all mathematical arguments, written this manuscript, and take full responsibility for the correctness, originality and content of the work.

	\section{Background and Preliminaries}\label{S:2}
	We begin with a brief review of some combinatorial concepts and the action of the Hecke algebra on Laurent polynomials, followed by the definitions of symmetric and non-symmetric polynomials. We adopt the notation of \cite[Section 5]{BHMPPS:25}.
	
	Throughout this paper, $\gamma,\eta$ denote weak compositions of length $n$, $\alpha,\beta$ represent compositions of length at most $n$, and $\lambda,\mu$ are taken to be partitions of length at most $n$. All polynomials considered are in the finite set of variables $X_n=(x_1,\ldots,x_n)$; for simplicity, we omit $X_n$ from the notation. We use $\langle\S_{*}\rangle f$ to denote the coefficient of $\S_{*}$ in the $\S$-expansion of $f$.
	
	\subsection{Compositions, partitions and polynomials}\label{ss:2.1}
	A {\em \color{blue} weak composition} $\gamma=(\gamma_1,\ldots,\gamma_n)$ of $N$, denoted by $\gamma\vDash _w N$, is a finite sequence of non-negative integers such that $\gamma_1+\cdots+\gamma_n=N$. The size and length of $\gamma$ are defined by $|\gamma|=N$ and $\ell(\gamma)=n$, respectively. Each $\gamma_i$ is called a {\em \color{blue} part} of $\gamma$, and the {\em \color{blue} reverse} of $\gamma$ is denoted by $\gamma^r=(\gamma_n,\ldots,\gamma_1)$. A weak composition $\alpha$ is called a {\em \color{blue}composition} if every part is positive, and we write $\alpha\vDash N$. A {\em\color{blue} partition} $\lambda=(\lambda_1\ge \cdots\ge \lambda_l>0)$ of $N$ is a weakly decreasing composition of $N$, denoted by $\lambda\vdash N$. Sometimes we write $\lambda=(\lambda_1\ge\cdots\ge\lambda_n\ge 0)$ when $\ell(\lambda)=l\le n$, with the convention that $\lambda_i=0$ for $\ell(\lambda)<i\le n$.  {When specifying a filling type, we write
		$\lambda^r=(0<\lambda_l\le\ldots\le\lambda_1)$.} Given a weak composition $\gamma$, {let $\gamma^+$ denote the composition obtained by deleting all zero parts of $\gamma$.
		Let $\lambda(\gamma)$ denote the sequence obtained by rearranging all parts of $\gamma$, including zeros, in weakly decreasing order, and write
		$\lambda^r(\gamma)=(\lambda(\gamma))^r$ for its reverse.}
	\begin{example}
		Let $\gamma=(1,0,2,2,4,5)$ be a weak composition of length $\ell(\gamma)=6$ and of size $|\gamma|=14$. Then {$\gamma^+=(1,2,2,4,5)$, $\lambda(\gamma)=(5,4,2,2,1,0)$ and $\lambda^r(\gamma)=(0,1,2,2,4,5)$.}
	\end{example}
	For a weak composition $\eta=(\eta_1,\ldots,\eta_n)\in\mathbb{N}^n$, the \emph{\color{blue} skyline diagram} consists of
	left-justified rows of boxes, with the $i$th row from the top containing $\eta_i$ boxes, in the {\em \color{blue} English convention}.
	The columns of the diagram are numbered from the left, and the diagram is augmented by a {\em\color{blue} basement}, consisting of an additional box in column $0$ of each row, marked with $\ast$.
	The skyline diagram with a basement is denoted by $\dg(\eta)$.
	
	For two weak compositions $\gamma,\eta\in\mathbb{N}^n$ satisfying $\gamma_i\leq\eta_i$ for all $i$, we simply write $\gamma\le \eta$.
	The \emph{\color{blue} skew skyline diagram} $\dg(\eta/\gamma)$ is obtained from the augmented skyline diagram
	$\dg(\eta)$ by treating the boxes belonging to $\dg(\gamma)$ as an extension of the basement and
	marking them with $\ast$.
	The boxes outside this extended basement are called \emph{\color{blue} ordinary boxes}; their coordinates are precisely $\{(i,k):1\leq i\leq n,\ \gamma_i<k\leq\eta_i\}$. See Figure \ref{fig:skyline} for an example.
	
	The {\em \color{blue} Young diagram} of $\lambda$, denoted by $\dg(\lambda)$, is an array of left-justified boxes with $\lambda_i$ boxes in row $i$ from the top.
	A box in the Young diagram of partition $\lambda$ has {\em \color{blue} coordinates $(i,j)$} if it is in the $i$th row from the top and the $j$th column from the left. Let $\lambda'$ be the {\em \color{blue} conjugate or transpose} of $\lambda$; that is, the Young diagram of $\lambda'$ is obtained from that of $\lambda$ by reflecting across the main diagonal (boxes with coordinates $(i,i)$). For two partitions $\lambda$ and $\mu$ with $\lambda_i\ge \mu_i$ for every $i$, the {\em \color{blue} skew diagram} of $\lambda/\mu$ is $\dg(\lambda/\mu)=\{(i,j) \,\vert\,i\ge 1, \,\mu_i+1\le j\le \lambda_i\}$. If $\mu=\emptyset$, then $\lambda=\lambda/\emptyset$.
	\begin{figure}[H]
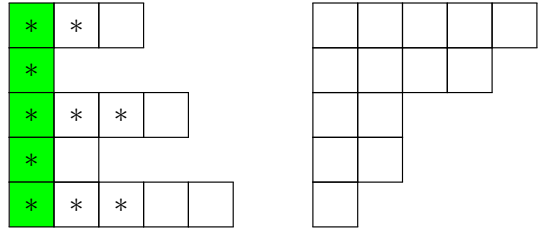

		\centering
		\begin{minipage}[H]{0.25\linewidth}
			\centering
			\[\begin{ytableau}
				*(green) * & \ast & \ \\
				*(green) * & \none \\
				*(green) * &\ast &  \ast & \ \\
				*(green) * & \ \\
				*(green) * & \ast & \ast & \ & \
			\end{ytableau}
			\]
		\end{minipage}
		\begin{minipage}[H]{0.25\linewidth}
			\centering
			\[\begin{ytableau}
				\ & \ & \ & \ & \ \\
				\ & \  &  \ & \ \\
				\ & \ \\
				\ & \ \\
				\ & \none & \none \\
			\end{ytableau}
			\]
		\end{minipage}
		\caption{A skew skyline diagram $\dg(\eta/\gamma)$ where $\eta=(2,0,3,1,4)$, $\gamma=(1,0,2,0,2)$ and the basement is colored green (left), and a Young diagram $\dg(\lambda)$ with $\lambda=(5,4,2,2,1)$ (right).}\label{fig:skyline}
	\end{figure}
	
	Let $\mathfrak{S}_n$ denote the symmetric group of $[n]=\{1,\ldots,n\}$, and let $s_{i}=(i\,i+1)$ be the {\em\color{blue} simple transposition} that swaps $i$ and $i+1$, for $1\le i<n$. It is well-known that $\mathfrak{S}_n$ is generated by these simple transpositions. An expression  $w=s_{i_{l}}\cdots s_{i_1}$ is said to be {\em \color{blue} reduced} if $\ell$ is minimal among all expressions for $w$. This minimal number $l$ is called the {\em\color{blue} length} of ${w}$ and is denoted by $\ell(w)$, which can be computed by counting the inversions of $w$, i.e.,
	\begin{align}\label{E:inv}
		\ell(w)=\inv(w)=|\{(i,j)\,\vert\,i<j \mbox{ and }\,w(i)>w(j)\}|.
	\end{align}
	For a Laurent polynomial $f\in \mathbb{Q}(t)[x_1^{\pm 1},\ldots,x_n^{\pm 1}]$, the transposition $s_i$ acts on $f$ by swapping the variables $x_i$ and $x_{i+1}$; namely,
	$s_i f(\ldots, x_i,x_{i+1},\ldots)=f(\ldots,x_{i+1},x_i,\ldots)$.
	The {\em \color{blue} Demazure-Lusztig operator} $T_i$ is defined by
	\begin{align}\label{Def:T}
		T_i=ts_i+\frac{(1-t)x_{i+1}}{x_{i+1}-x_i}(s_i-1)=t+\frac{tx_i-x_{i+1}}{x_i-x_{i+1}}(s_i-1).
	\end{align}
	for $1\le i<n$. These operators give rise to an action of the Hecke algebra $\H(\mathfrak{S}_n)$ on the ring of Laurent polynomials $\mathbb{Q}(t)[x_1^{\pm 1},\ldots,x_n^{\pm 1}]$ and satisfy the following relations:
	\vspace{-2mm}
	\begin{enumerate}
		\item Quadratic relation: $(T_i-t)(T_i+1)=0$.
		\item Braid relations: $T_iT_{i+1}T_i=T_{i+1}T_iT_{i+1}$ and $T_iT_j=T_jT_i$ for $|i-j|>1$.
	\end{enumerate}
	\vspace{-2mm}
	Meanwhile, the braid relations (2) imply that the operator
	\begin{align}\label{E:Tw}
		T_{w}=T_{i_1}T_{i_2}\cdots T_{i_{l}}
	\end{align}
	for any reduced expression $w=s_{i_1}s_{i_2}\cdots s_{i_l}$, is well-defined in the sense that it is independent of the chosen reduced word. The set  $\{T_w\,\vert\,w\in \mathfrak{S}_n\}$ forms a $\mathbb{Q}(t)$-basis of the Hecke algebra $\H(\mathfrak{S}_n)$.  We denote by $\overline{T}_i$ the simultaneous substitutions $t\mapsto t^{-1}$ and $x_i\mapsto x_i^{-1}$ in the right-hand side of (\ref{Def:T}). Using this convention, we have $\overline{T_i}=T_{i}^{-1}$, which together with (\ref{E:Tw}) implies
	\begin{align*}
		\overline{T_w}=\overline{T_{i_1}}\,\overline{T_{i_2}}\cdots \overline{T_{i_{l}}}=(T_{i_1})^{-1}(T_{i_2})^{-1}\cdots (T_{i_l})^{-1}=(T_{w^{-1}})^{-1}.
	\end{align*}
	For a positive integer $n$, we define the {\em \color{blue}symmetrizer}
	\begin{align*}
		\CMcal{S}_{n}=\sum_{w\in \mathfrak{S}_n}T_{w}.
	\end{align*}
	{For a partition $\lambda$ with $\ell(\lambda)\le n$, we write $\lambda=(\lambda_1\ge \cdots\ge\lambda_n\ge 0)$,} and let $m_i=m_i(\lambda)$ denote the multiplicity of part $i$ in $\lambda$ for $i\ge 0$. Let
	\begin{align}\label{E:vlam}
		v_{\lambda}(t)=\prod_{i\ge 0} [m_i]_t!,\,\mbox{ where }\,  [m]_t!=\prod_{1\le k\le m}[k]_t\,\, \mbox{ with }
		[k]_t=\frac{1-t^k}{1-t}=\sum_{i=0}^{k-1}t^i
	\end{align}
	for $k\ge 1$ and $[0]_t=1$. Moreover,
	let $x^{\lambda}=x_1^{\lambda_1}\cdots x_n^{\lambda_n}$ be the monomial corresponding to $\lambda=(\lambda_1\ge \cdots\ge \lambda_n\ge 0)$.
	\begin{definition}[{\em \color{blue} Hall-Littlewood $P$-polynomials}]\label{Def:P}
		The Hall-Littlewood $P$-polynomial, in the variables $X_n=(x_1,\ldots,x_n)$ and indexed by a partition $\lambda$  {with $\ell(\lambda)\le n$}, is a symmetric polynomial given by
		\begin{align}\label{E:defP}
			P_{\lambda}(t)=P_{\lambda}(X_n;t)=\frac{\CMcal{S}_{n}(x^{\lambda})}{v_{\lambda}(t)}=\frac{1}{v_{\lambda}(t)}\sum_{w\in \mathfrak{S}_n} w\left(x^{\lambda}\prod_{i<j}\frac{1-tx_j/x_i}{1-x_j/x_i}\right),
		\end{align}
		where the last equality follows from Lusztig's identity \cite[Theorem 6.6]{Lusztig}. These polynomials $P_{\lambda}(t)$ for partitions $\lambda$ of length at most $n$ form a basis of the symmetric polynomials in $\mathbb{Q}(t)[x_1,\ldots,x_n]$. In particular, $P_{\lambda}(0)=s_{\lambda}$ is the {\color{blue} Schur function}.
	\end{definition}
	\begin{example}
		For $n=3$ and $\lambda=(2,1)$, we have $v_{\lambda}(t)=1$ and
		\begin{align*}
			P_{21}(t)=\CMcal{S}_3(x_1^2x_2)&=(1+T_1+T_2+T_1T_2+T_2T_1+T_1T_2T_1)(x_1^2x_2)\\
			&=\sum_{w\in \mathfrak{S}_3}w\left(x_1^2x_2\frac{(x_1-tx_2)(x_1-tx_3)(x_2-tx_3)}{(x_1-x_2)(x_1-x_3)(x_2-x_3)}\right)\\
			&=x_1^2x_2+x_1^2x_3+x_2^2x_3+x_1x_2^2+x_1x_3^2+x_2x_3^2+(2-t-t^2)x_1x_2x_3.
		\end{align*}
	\end{example}
	For any permutation $w\in \mathfrak{S}_n$ and any weak composition $\gamma\in \mathbb{N}^n$, we define $(w\gamma)_i=\gamma_{w^{-1}(i)}$ for $1\le i\le n$, namely $w\gamma=(\gamma_{w^{-1}(1)},\ldots,\gamma_{w^{-1}(n)})$.
	
	\begin{definition}[{\em \color{blue} Non-symmetric Hall-Littlewood polynomials}] \label{Def:F}
		For a sequence $\gamma\in \mathbb{Z}^{n}$, let $\mu=\lambda(\gamma)$. Then
		the non-symmetric Hall-Littlewood polynomials are given by
		\begin{align*}
			F_{\gamma}(t^{-1})&=F_{\gamma}(X_n;t^{-1})=\overline{T_v}\,x^{\mu^r},\\
			\E_{\gamma}(t)&=\E_{\gamma}(X_n;t)=T_{w}\,x^{\mu},
		\end{align*}
		where $v$ and $w$ are the unique permutations of minimal length such that $\gamma=v\mu^r=w\mu$, respectively. The sets $\{F_{\gamma}(t)\}_{\gamma\in \mathbb{N}^n}$ and $\{\E_{\gamma}(t)\}_{\gamma\in \mathbb{N}^n}$ each form a basis of the polynomial ring $\mathbb{Q}(t)[x_1,\ldots,x_n]$. In particular, $\E_{\gamma}(0)=\A_{\gamma}$ is the {\color{blue} $t$-Demazure atom}.
	\end{definition}

	\begin{proposition}\cite[Proposition 28, equation (37)]{Alex16}
	Let $\lambda=\lambda(\gamma)$, and let $\tau$ be the shortest permutation such that $\gamma=\tau \lambda$. Then,
	\begin{align*}
		\A_{\gamma}(X_n;t)=T_{\tau}x^{\lambda}=\E_{\gamma}(X_n;t).
	\end{align*}
\end{proposition}
	
	\begin{example}
		If {$\gamma=(1,2,1,0)$} has length $n=4$, then $\mu=\lambda(\gamma)=(2,1,1,0)$ and $\mu^r=(0,1,1,2)$. Note that $v=4132=s_3s_2s_1s_3$ is the shortest permutation satisfying $\gamma=v\mu^r$, and hence
		\begin{align*}
			F_{1210}(t^{-1})=\overline{T_3}\,\overline{T_2}\,\overline{T_1}\,\overline{T_3}\,(x_2x_3x_4^2)
			=x_1x_2^2x_3+t^{-1}(1-t^{-1})x_1x_2x_3x_4.
		\end{align*}
		Equivalently, $F_{1210}(t)=x_1x_2^2x_3+t(1-t)x_1x_2x_3x_4$.
	\end{example}
	Let $w_0=n\,n-1\,\ldots\,1$ be the longest permutation of $[n]$. For any finite sequence $a=(a_1,\ldots,a_n)$, note that $a^r=w_0\,a=(a_n,\ldots,a_1)$. The lemma below states that $F_{\gamma}(t)$ and $\E_{\gamma}(t)$ are equal, up to the reversal of the weak composition $\gamma$ and of the variables $X_n$.
	\begin{lemma}\label{E:etof}
		For every weak composition $\gamma\in \mathbb{N}^n$, we have
		\begin{align}\label{E:EtoF}
			\E_{\gamma}(t)=\E_{\gamma}(X_n;t)=F_{\gamma^{r}}(X_n^r;t).
		\end{align}
	\end{lemma}
	\begin{proof}
		Let $\delta=w_0\,\gamma=\gamma^r$. Then $\lambda^r(\delta)=\lambda^r(\gamma)=\mu^r$. Take $w$ to be the shortest permutation satisfying $\gamma=w\mu$ and choose $v=w_0 w w_0$. It follows that
		\begin{align*}
			v\mu^r=(w_0 w w_0)\mu^r=(w_0 w w_0)w_0\mu=w_0\gamma=\delta.
		\end{align*}
		Since $w_0s_iw_0=s_{n-i}$ for $1\le i<n$, we have $\ell(w)=\ell(v)$, thus implying that $v$ is the unique permutation of minimal length such that $\delta=v\mu^r$. By Definition \ref{Def:F}, we obtain
		\begin{align}\label{E:F2}
			F_{\delta}(X_n^r;t^{-1})=w_0F_{\delta}(X_n;t^{-1})=w_0\overline{T_v}(t)x^{\mu^r}=(w_0\,\overline{T_v}(t)\,w_0)x^{\mu}.
		\end{align}
		We claim that $w_0\,\overline{T_v}(t)\,w_0=T_{w}(t^{-1})$. Since $w_0s_iw_0=s_{n-i}$ for $1\le i<n$, we find that
		\begin{align}\label{E:Ti}
			w_0\overline{T_i}(t)w_0=t^{-1}s_{n-i}+\frac{(1-t^{-1})x_{n-i+1}}{x_{n-i+1}-x_{n-i}}(s_{n-i}-1)=T_{n-i}(t^{-1}).
		\end{align}
		For any reduced expression $w=s_{i_1}\cdots s_{i_r}$ of $w$, we have
		\begin{align*}
			v=w_0\,w\,w_0=(w_0s_{i_1}w_0)\cdots (w_0s_{i_r}w_0)=s_{n-i_1}\cdots s_{n-i_r}.
		\end{align*}
		Consequently, it follows from (\ref{E:Ti}) that
		\begin{align*}
			w_0\overline{T_v}(t)w_0=(w_0\overline{T}_{n-i_1}(t)w_0)\cdots (w_0\overline{T}_{n-i_r}(t)w_0)=T_{i_1}(t^{-1})\cdots T_{i_r}(t^{-1})=T_{w}(t^{-1}).
		\end{align*}
		Substituting this into (\ref{E:F2}) gives $F_{\delta}(X_n^r;t^{-1})=T_{w}(t^{-1})x^{\mu}=\E_{\gamma}(t^{-1})$, completing the proof.
	\end{proof}	
	It follows from Definitions \ref{Def:P} and \ref{Def:F} that the non-symmetric Hall-Littlewood polynomials $\E_{\gamma}(t)$ naturally refine the Hall-Littlewood polynomials $P_{\lambda}(t)$, as stated below. Since $P_{\lambda}(t)$ and $\E_{\gamma}(t)$ are specializations of symmetric and non-symmetric Macdonald polynomials, the following proposition is also an immediate consequence of \cite{Mac94}\cite[Lemma 2.5 (a)]{Marshall: 99}\cite[Proposition 5.3.1]{HHL08}\cite[Corollary 30]{Alex16}. To keep the paper self-contained, we give a proof of this refinement.
	\begin{proposition} \label{prop:PtoE}
		For a partition $\lambda=(\lambda_1\ge \cdots\ge \lambda_{n}\ge0)$, we have
		\begin{align}\label{E:pe11}
			P_{\lambda}(t)&=\sum_{\lambda(\gamma)=\lambda}\E_{\gamma}(t).
		\end{align}
	\end{proposition}
	\begin{proof}
		Let $G_{\lambda}=\{w\in\mathfrak{S}_n\,\vert\,w\lambda=\lambda\}$ be the stabilizer group of $\lambda$. For each weak composition $\gamma\in \mathbb{N}^n$, let $w_{\gamma}$ be the unique permutation of minimal length such that $\gamma=w_{\gamma}\lambda$, then the set
		\begin{align*}
			H_{\lambda}=\{w_{\gamma}\,\vert\,\lambda(\gamma)=\lambda\}
		\end{align*}
		is a set of minimal coset representatives in the cosets $\mathfrak{S}_n/G_{\lambda}$. Therefore,
		\begin{align}\label{E:cosetdec}
			\mathfrak{S}_n=\bigsqcup_{w\in H_{\lambda}}wG_{\lambda}.
		\end{align}
		For any $w\in H_{\lambda}$ and $\sigma\in G_{\lambda}$, we have $\ell(w\sigma)=\ell(w)+\ell(\sigma)$ and hence $T_{w\sigma}=T_{w}T_{\sigma}$. We claim that
		\begin{align}\label{E:claim1}
			\sum_{\sigma\in G_{\lambda}}T_{\sigma}x^{\lambda}=\sum_{\sigma\in G_{\lambda}}t^{\ell(\sigma)}x^{\lambda}=v_\lambda(t)x^{\lambda}.
		\end{align}
		For each $s_i\in G_{\lambda}$, we have $s_i\lambda=\lambda$, thus $T_ix^{\lambda}=tx^{\lambda}$ by (\ref{Def:T}) and hence $T_{\sigma}x^{\lambda}=t^{\ell(\sigma)}x^{\lambda}$ for any $\sigma\in G_{\lambda}$. Recall that $m_i=m_i(\lambda)$ counts the multiplicity of part $i$ in the partition $\lambda$ for $i\ge 0$ and
		\begin{align*}
			\sum_{s\in \mathfrak{S}_m}t^{\ell(s)}=\sum_{s\in \mathfrak{S}_m}t^{\inv(s)}=[m]_t!
		\end{align*}
		where the inversion statistic is defined by (\ref{E:inv}).
		{Since $G_\lambda\cong\prod_{i\ge0}\mathfrak S_{m_i}$, we find
			\begin{align*}
				\sum_{\sigma\in G_{\lambda}}t^{\ell(\sigma)}=\prod_{i\ge 0}\sum_{s\in \mathfrak{S}_{m_i}}t^{\ell(s)}=\prod_{i\ge 0}\sum_{s\in \mathfrak{S}_{m_i}}t^{\mathrm{inv}(s)}
				=\prod_{i\ge 0}[m_i]_t!=v_{\lambda}(t).
		\end{align*} }
		This proves the claim (\ref{E:claim1}). Using the coset decomposition (\ref{E:cosetdec}) and (\ref{E:claim1}), we arrive at
		\begin{align*}
			\CMcal{S}_n(x^{\lambda})=\sum_{w\in H_{\lambda}}\sum_{\sigma\in G_{\lambda}}T_{w\sigma}x^{\lambda}=\sum_{w\in H_{\lambda}}T_w\sum_{\sigma\in G_{\lambda}}T_{\sigma}x^{\lambda}=v_{\lambda}(t)\sum_{\lambda(\gamma)=\lambda}\E_{\gamma}(t),
		\end{align*}
		which completes the proof of (\ref{E:pe11}) by (\ref{E:defP}).
	\end{proof}
	\begin{example}
		For $\lambda=(2,1,0)$ with $n=3$, we have
		\begin{align*}
			\E_{210}(t)&=x_1^2x_2,\quad \E_{120}(t)=x_1x_2^2,\\
			\E_{201}(t)&=x_1^2x_3,\quad \E_{102}(t)=x_1x_3^2+(1-t)x_1x_2x_3,\\
			\E_{021}(t)&=x_2^2x_3+(1-t)x_1x_2x_3,\quad \E_{012}(t)=x_2x_3^2+t(1-t)x_1x_2x_3.
		\end{align*}
		Consequently, $P_{21}(t)=\E_{210}(t)+\E_{120}(t)+ \E_{201}(t)+\E_{102}(t)+\E_{021}(t)+\E_{012}(t)=(x_1^2x_2+x_1x_2^2+x_1^2x_3+x_1x_3^2+x_2^2x_3+x_2x_3^2)+(2-t-t^2)x_1x_2x_3$.
	\end{example}
	
	\subsection{Plethystic substitution}
	Let $A$ be a Laurent series in indeterminates $a_1,a_2,\ldots$ over $\mathbb{Q}(t)$, and define $p_k[A]$ by substituting $a_i\mapsto a_i^k$ in $A$. The {\em \color{blue} plethystic substitution} of $A$ into a symmetric function $f$, denoted by $f[A]$, is obtained by extending the specialization $p_k\mapsto p_k[A]$ to $f$. If $A$ is a sum of indeterminates, $A=a_1+\cdots+a_n$, then $p_k[A]=p_k(a_1,\ldots,a_n)$ and hence for every symmetric function $f$, we have $f[A]=f(a_1,\ldots,a_n)$. In plethystic notation, it is convenient to define the {\em \color{blue} plethystic exponential} by
	\begin{align*}
		\Exp[X]=\exp\left(\sum_{k=1}^{\infty}\frac{p_k[X]}{k}\right),\,\,\mbox{ where }\,\,X=\sum_{i\ge 1}x_i.
	\end{align*}
	The following lemma gives the nonplethystic expression for $\Exp[XY(1-t)]$.
	\begin{lemma}\cite[Section 2]{Haiman:99}
		For any $X=(x_1,x_2,\ldots)$ and $Y=(y_1,y_2,\ldots)$,
		\begin{align}\label{E:omegaXY}
			\Exp[XY(1-t)]=\prod_{i,j\ge 1}\frac{1-tx_iy_j}{1-x_iy_j}.
		\end{align}
	\end{lemma}
	\begin{proof}
		Since $p_k[X+Y]=p_k[X]+p_k[Y]$ and $p_k[-X]=-p_k[X]$ by definition, it is immediate that
		\begin{align}\label{E:omega}
			\Exp[X+Y]=\Exp[X]\,\Exp[Y] \quad \mbox{ and }\quad \Exp[-X]=\Exp[X]^{-1}.
		\end{align}
		By noting that $$\Exp(x)=\Exp[x]=\exp\left(\sum_{k=1}^{\infty}\frac{x^k}{k}\right)=(1-x)^{-1}$$ for a single variable $x$, we can express $\Exp[X]$ in the basis of complete homogeneous symmetric functions. More explicitly, it follows from (\ref{E:omega}) that
		\begin{align*}
			\Exp[X]&=\prod_{i\ge 1}\Exp[x_i]=\prod_{i\ge 1}(1-x_i)^{-1}
		\end{align*}
		Together with the second equality in (\ref{E:omega}), we find $\Exp[XY(1-t)]=\Exp[XY]\Exp[tXY]^{-1}$, thus yielding (\ref{E:omegaXY}), as desired.
	\end{proof}
	\begin{definition}[{\em \color{blue} Dual bases}]\label{Def: dual}
		Two bases $\{u_{\lambda}(X;t)\}$ and $\{v_{\mu}(X;t)\}$
		{of the algebra of symmetric functions over $\mathbb{Q}(t)$, in infinitely many variables $X=(x_1,x_2,\ldots)$,
			are said to be mutually dual with respect to the Hall inner product $\langle\cdot,\cdot\rangle_t$ if
			$\langle u_\lambda(X;t),v_\mu(X;t)\rangle_t=\delta_{\lambda\mu}$.
			Here the inner product is defined on the power-sum basis by }
		\begin{align*}
			\langle p_{\lambda},p_{\mu}\rangle_t=z_{\lambda}\delta_{\lambda\mu}\prod_{i=1}^{\ell(\lambda)}(1-t^{\lambda_i})^{-1}, \,\, \mbox{ where }\, z_{\lambda}=\prod_{i\ge 1}i^{m_i}m_i!
		\end{align*}
		with $m_i=m_i(\lambda)$ counting the multiplicity of $i$ in $\lambda$. Equivalently, these two bases are dual if and only if
		\begin{align}\label{E:uv}
			\sum_{\lambda}u_{\lambda}(X;t)v_{\lambda}(Y;t)=\Exp[XY(1-t)]=\prod_{i,j\ge 1}\frac{1-tx_iy_j}{1-x_iy_j},
		\end{align}
		as follows from \cite[Chapter VI, (2.7)]{Mac95} and (\ref{E:omegaXY}).
	\end{definition}
	Now we introduce the dual basis of the Hall-Littlewood $P$-polynomials $P_{\lambda}(t)$.
	\begin{definition}[{\em \color{blue} Modified Hall-Littlewood polynomials}]
		For a partition $\lambda=(\lambda_1\ge \cdots\ge \lambda_{n}\ge 0)$, the Hall-Littlewood $Q$-polynomial $Q_{\lambda}(t)$ is given by {
			\begin{align*}
				Q_{\lambda}(t)=(1-t)^{l}\frac{v_{\lambda}(t)}{[n-l]_t!}P_{\lambda}(t),
		\end{align*} }
		where $l$ is the number of non-zero parts of $\lambda$ and $v_{\lambda}(t)\in \mathbb{Z}[t]$ is defined in (\ref{E:vlam}).
		The modified Hall-Littlewood polynomial $Q'_{\lambda}(t)$ is defined by
		\begin{align*}
			Q'_{\lambda}(t)=Q'_{\lambda}(X_n;t)=Q_{\lambda}\left[\frac{X_n}{1-t};t\right].
		\end{align*}
	\end{definition}
	\begin{theorem}[{\em \color{blue} Hall-Littlewood Cauchy identity}]\cite[Chapter III, a specialization of (4.4) at $q=0$]{Mac95}\label{T:CauchyHL}
		The two symmetric bases $\{P_{\lambda}(t)\}$ and $\{Q_{\lambda}(t)\}$ are dual bases, namely, by (\ref{E:uv}),
		\begin{align}\label{E:Q'P}
			\sum_{\lambda}Q'_{\lambda}(X_n;t)P_{\lambda}(Y_n;t)=\Exp[X_nY_n],
		\end{align}
		where the sum is taken over all partitions $\lambda=(\lambda_1\ge \cdots \ge \lambda_n\ge 0)$.
	\end{theorem}
	By (\ref{E:omegaXY}), we have {$\Exp[X_nY_n]=\prod_{i,j=1}^n(1-x_i y_j)^{-1}$, which is symmetric separately in $X_n$ and $Y_n$.}
	{Expanding $\Exp[X_nY_n]\E_\gamma(Y_n;t)$ in the $\E$-basis with respect to $Y_n$ gives coefficients that are symmetric
		in $X_n$, since the product is symmetric in $X_n$. We may therefore expand these coefficients in the modified
		Hall--Littlewood $Q'$-basis as follows.}
	\begin{definition}[{\em \color{blue}Coefficients $K_{\gamma\eta}(X_n;t)$}]\label{Def:Ea}
		For two weak compositions $\gamma,\eta\in \mathbb{N}^n$, set
		\begin{align}
			K_{\gamma\eta}(X_n;t)&=\langle\E_{\eta}(Y_n;t)\rangle(\Exp[X_nY_n]\E_{\gamma}(Y_n;t))\notag\\
			&=\sum_{\lambda}a_{\gamma\lambda}^{\eta}(t)Q'_{\lambda}(X_n;t).\label{E:Ka}
		\end{align}
		The second equality follows by multiplying both sides of (\ref{E:Q'P}) by $\E_{\gamma}(Y_n;t)$, employing (\ref{E:defa}) and interchanging the  sums degree by degree.
	\end{definition}
	\subsection{LLT polynomials}  Lascoux, Leclerc and Thibon \cite{LLT:97} introduced  LLT polynomials, motivated by the study of Fock space representations of quantum affine Lie algebras.
	
	A {\em \color{blue} semistandard Young tableau} (SSYT) of skew shape $\lambda/\mu$ is a function $T:\dg(\lambda/\mu)\rightarrow \mathbb{N}^+=\{1,2,\ldots\}$ such that $T(i,j)\le T(i,j+1)$ and $T(i,j)<T(i+1,j)$, namely, each row is weakly increasing from the left, and each column is strictly increasing from the top.  A {\em \color{blue} standard Young tableau} (SYT) of shape $\lambda/\mu$ is a function $T:\dg(\lambda/\mu)\rightarrow \{1,\ldots,|\lambda|-|\mu|\}$ such that $T(i,j)<T(i,j+1)$, $T(i,j)<T(i+1,j)$ and each integer from $1$ to $|\lambda|-|\mu|$ is used exactly once. We use $\SSYT(\lambda/\mu)$ and $\SYT(\lambda/\mu)$ to denote the sets of semistandard and standard Young tableaux of shape $\lambda/\mu$, respectively.
	
	\begin{definition}[{\em \color{blue} LLT polynomial}]\label{Def:attack}
		For a box $u=(i,j)$ of a skew diagram, let $c(u)=j-i$ be the {\color{blue} content} of $u$. Given a tuple $\bnu=(\nu^1,\ldots,\nu^{l})$ of skew diagrams, two boxes $x\in \nu^{i}$ and $y\in \nu^{j}$ form {\color{blue} an attacking pair} if and only if
		\begin{enumerate}
			\item $i<j$ and $c(x)=c(y)$, or
			\item $i>j$ and $c(x)+1=c(y)$.
		\end{enumerate}
		Let $\SSYT(\bnu)=\{\bT=(T^1,\ldots,T^l)\,\vert\, T^i\in \SSYT(\nu^i)\}$ be the set of tuples of semistandard Young tableaux of shape $\bnu$. {\color{blue} An attacking inversion} in the tuple $\bT$ is an attacking pair $(x,y)$ with $x\in \nu^i$ and $y\in \nu^j$ such that $T^i(x)>T^j(y)$. The number of attacking inversions in $\bT$ is denoted by $\inv(\bT)$. For each tableau $T^i$, let
		\begin{align*}
			x^{T^i}=\prod_{u\in \dg(\nu^i)}x_{T^i(u)}
		\end{align*}
		be the monomial corresponding to $T^i$, and let $x^{\bT}=x^{T^1}x^{T^2}\cdots x^{T^l}$. Then, the {\color{blue} LLT polynomial} $\CMcal{G}_{\bnu}(X;t)$ is defined by
		\begin{align*}
			\CMcal{G}_{\bnu}(X;t)=	\CMcal{G}_{\bnu}(x_1,x_2,\ldots;t)=\sum_{\bT\in \mathrm{SSYT}(\bnu)}t^{\inv(\bT)}x^{\bT}.
		\end{align*}
		In particular, if $\bnu=(\lambda/\mu)$ has only one skew diagram, then $\CMcal{G}_{\bnu}(X;t)=\CMcal{G}_{\bnu}(X;1)=s_{\lambda/\mu}(X)$ is the {\color{blue} skew Schur function} of skew diagram $\lambda/\mu$.
	\end{definition}
	\begin{example}
		For $\bnu=((2),(1))$, we have $\CMcal{G}_{\bnu}(X;t)=s_3(X)+ts_{21}(X)$.
	\end{example}

	Lascoux, Leclerc and Thibon \cite{LLT:97} proved that the LLT polynomials are symmetric. About a decade later, Grojnowski and Haiman \cite{GH:06} established the {\em \color{blue} Schur-positivity} of LLT polynomials, that is, each LLT polynomial has a positive expansion in terms of Schur functions. Their proof used affine Hecke algebras and Kazhdan-Lusztig theory.
	
	Let $\omega$ denote the endomorphism on the algebra of symmetric functions defined by $\omega(e_n)=h_n$ for $n\ge 1$. Since $\omega$ preserves multiplication, it follows that $\omega(e_{\lambda})=h_{\lambda}$. It was established in \cite[Theorem 7.6.1]{Stanley} that $\omega$ is an involution, meaning that $\omega(h_n)=e_n$ and consequently $\omega(h_{\lambda})=e_{\lambda}$. 
	For a tuple $\bnu$ of skew diagrams, let $\bnu'$ be obtained from $\bnu$ by transposing each skew diagram in $\bnu$ and reversing the tuple.
	Haglund, Haiman and Loehr \cite[Lemma 10.1 and its remark]{HHL04} proved that
	\begin{align}\label{E:omegaG}
		\CMcal{G}_{\bnu'}(X;t)=t^{m(\bnu)}\omega \CMcal{G}_{\bnu}(X;t^{-1})
	\end{align}
	where $m(\bnu)$ is the number of attacking pairs of boxes in $\bnu$.
	\begin{definition}[{\em \color{blue} $q$-Whittaker polynomial}]
		The $q$-Whittaker polynomial $W_{\lambda}(X;t)$ is obtained by applying $\omega$ to the modified Hall-Littlewood polynomial $Q'_{\lambda'}(X;t)$; that is,
		\begin{align}\label{E:qWhittaker}
			W_{\lambda}(X;t)=\omega\, Q'_{\lambda'}(X;t).
		\end{align}
	\end{definition}
	\begin{remark}
		The $q$-Whittaker polynomial $W_{\lambda}(X;q)$ is a specialization of the Macdonald polynomials $P_{\lambda}(X;q,t)$ at $t=0$, whereas the Hall-Littlewood polynomial $P_{\lambda}(X;t)$ is a specialization of the Macdonald polynomials $P_{\lambda}(X;q,t)$ at $q=0$.
	\end{remark}
	\subsection{From vertical strips to Schr\"oder paths}
	A {\em \color{blue} vertical-strip LLT polynomial} $\CMcal{G}_{\bnu}(X;t)$ is an LLT polynomial indexed by a tuple $\bnu$ of columns.
	
	
	Let $|\bnu|$ denote the total number of boxes in $\bnu$. Every vertical-strip LLT polynomial $\G_{\bnu}(X;t)$ of degree $|\bnu|$ can be realized as a symmetric polynomial indexed by a Schr\"{o}der path from $(0,0)$ to $(|\bnu|,|\bnu|)$; see \cite[Proposition 3.4]{AP} and \cite[Remark 3.6]{CM:18}. To this end, we introduce Schr\"oder paths and decorated unit-interval graphs.
	\begin{definition}[{\em \color{blue}Schr\"{o}der path}]
		A Schr\"oder path $p$ of length $n$  is a lattice path  from $(0,0)$ to $(n,n)$ with the north step
		$\mathrm N=(0,1)$, the east step $\mathrm E=(1,0)$ and the diagonal step $\mathrm D=(1,1)$ that stays weakly above the main diagonal $x=y$ and has no diagonal step $(1,1)$ lying on the main diagonal.
		
		Each Schr\"oder path $p$ can be represented as a word in $\mathrm{N}$, $\mathrm{E}$ and $\mathrm{D}$. See Figure \ref{fig:path1}. A {\color{blue} Dyck path} is a Schr\"oder path without any diagonal step $\mathrm D=(1,1)$.
		
	\end{definition}
	\begin{definition}[{\em \color{blue} Decorated unit-interval graph}]
		A {\color{blue} unit-interval graph} is a graph $\Gamma=([n],E)$ with vertex set $[n]$ and edge set $E$ such that $uw\in E$ implies that $uv,vw\in E$ for all $u<v<w$. A {\color{blue} decorated unit-interval graph} $([n],E,S)$ is a  unit-interval graph $\Gamma=([n],E)$ together with a subset $S\subseteq E$ of {\color{blue} strict edges}. The remaining edges in $E-S$ are {\color{blue} non-strict edges}.
	\end{definition}
	Throughout this paper, all edges are written as $ij$ with $i<j$. Given a Schr\"oder path $p$ of length $n$, we associate with it a decorated unit-interval graph $\Gamma_p = ([n], E, S)$ as follows. For $u,v\in [n]$ with $u < v$, there is a non-strict edge $uv\in E-S$ if and only if there is a square in column $u$ and row $v$ below the path $p$; there is a strict edge $uv\in S$ if and only if
	$(u,v)$ is the endpoint of a diagonal step of $p$. See Figure \ref{fig:path1} for an example.
	
	If $S=\emptyset$, then $\Gamma_p = ([n], E, \emptyset)$ reduces to a unit-interval graph, which uniquely corresponds to a Dyck path of length $n$.  A {\em \color{blue} coloring} of $\Gamma_p=([n],E,S)$ is a function $\kappa:[n]\to\mathbb{N}^{+}$ such that $\kappa(u)<\kappa(v)$ for all $u,v\in [n]$ such that $uv\in S$ with $u<v$. The number of {\em \color{blue} ascents} of coloring $\kappa$ is defined by
	\begin{align}\label{eq:asc}
		\asc(\kappa)=\#\{uv\in E-S:\kappa(u)<\kappa(v)\},
	\end{align}
	and the corresponding generating function is given by
	\begin{align}\label{eq:coloring1}
		G_p(X;t)=
		\sum_{\kappa}
		t^{\asc(\kappa)}\prod_{1\le i\le n}x_{\kappa(i)},
	\end{align}
	where the sum is taken over all colorings of $\Gamma_p$. Indeed, the polynomial $G_p(X;t)$ indexed by Schr\"oder paths equals the LLT polynomial indexed by vertical strips.
	
	\begin{proposition}\cite[Proposition 3.4]{AP}\cite[Remark 3.6]{CM:18}\label{prop:path-vertical}
		Let $\bnu$ be a tuple of columns with $|\bnu|=n$ boxes in total; then mark the boxes $u$ in $\bnu$ by integers from $1$ to $n$ in decreasing order
		of $c(u)$, with boxes $u$ of the same content $c(u)$ ordered by decreasing column index. Then there exists a Schr\"oder path $p=p(\bnu)$ such that
		\begin{align*}
			\G_{\bnu}(X;t)=G_p(X;t),
		\end{align*}
		and the corresponding decorated unit-interval graph $\Gamma_p=([n],E,S)$ satisfies
		\[
		\begin{aligned}
			E-S
			&=\{pq:p<q\,\mbox{ and }\,
			(q,p)\text{ is an attacking pair}\},\\
			S
			&=\{pq: p<q\mbox{ and }\,
			\raisebox{2pt}{\begin{ytableau}[] p \\ q \end{ytableau}}\, \text{ in a column of $\bnu$}\}.
		\end{aligned}
		\]
	\end{proposition}

	\begin{example}
		Let $\bnu=\bigl((1,1)/(1),\,(1,1),\,(1,1),\,(1,1,1)/(1,1)\bigr)$.
		The labeling on the boxes of $\bnu$ is shown on the left of Figure \ref{fig:path1}.
		The contents $c(u)$ for boxes $u$ are $c(1)=c(2)=0$, $c(3)=c(4)=c(5)=-1$, and $c(6)=-2$. Proposition \ref{prop:path-vertical} shows that
		the decorated unit-interval graph $\Gamma_p=([6],E,S)$ has the non-strict edge set and the strict-edge set
		\begin{align*}
			E-S=\{12, 23, 34, 35, 36, 45, 46, 56\},\,\, S=\{13, 24\},
		\end{align*}
		respectively. See Figure \ref{fig:path1}.
	\end{example}
	\begin{figure}
		\begingroup
		\setlength{\unitlength}{5mm}
		\begin{center}
			\begin{picture}(23,7.6)(-0.4,-1.2)
				\linethickness{0.4pt}
				\put(1.65,3.625){\framebox(1.05,0.8){\large $2$}}
				\put(3.30,3.625){\framebox(1.05,0.8){\large $1$}}
				\put(0.00,2.600){\framebox(1.05,0.8){\large $5$}}
				\put(1.65,2.600){\framebox(1.05,0.8){\large $4$}}
				\put(3.30,2.600){\framebox(1.05,0.8){\large $3$}}
				\put(4.95,1.575){\framebox(1.05,0.8){\large $6$}}
				
				\put(7.6,0){%
					\begin{picture}(0,0)
						{\color{gray!65}\linethickness{0.3pt}
							\multiput(0,0)(1,0){7}{\line(0,1){6}}
							\multiput(0,0)(0,1){7}{\line(1,0){6}}
						}
						
						\put(0.5,0.5){\makebox(0,0){$1$}}
						\put(1.5,1.5){\makebox(0,0){$2$}}
						\put(2.5,2.5){\makebox(0,0){$3$}}
						\put(3.5,3.5){\makebox(0,0){$4$}}
						\put(4.5,4.5){\makebox(0,0){$5$}}
						\put(5.5,5.5){\makebox(0,0){$6$}}
						\put(0.5,1.5){\makebox(0,0){$12$}}
						\put(1.5,2.5){\makebox(0,0){$23$}}
						\put(2.5,3.5){\makebox(0,0){$34$}}
						\put(2.5,4.5){\makebox(0,0){$35$}}
						\put(3.5,4.5){\makebox(0,0){$45$}}
						\put(2.5,5.5){\makebox(0,0){$36$}}
						\put(3.5,5.5){\makebox(0,0){$46$}}
						\put(4.5,5.5){\makebox(0,0){$56$}}
						
						\linethickness{1.2pt}
						\polyline(0,0)(0,2)(1,3)(2,4)(2,6)(6,6)
						\put(0,0){\circle*{0.18}}
						\put(0,2){\circle*{0.18}}
						\put(1,3){\circle*{0.18}}
						\put(2,4){\circle*{0.18}}
						\put(2,6){\circle*{0.18}}
						\put(6,6){\circle*{0.18}}
						
						\linethickness{1.1pt}
						\Line(12.5,0.401924)(14,3)
						\Line(14,3)(12.5,5.598076)
						\Line(12.5,5.598076)(9.5,5.598076)
						\Line(12.5,5.598076)(8,3)
						\Line(12.5,5.598076)(9.5,0.401924)
						\Line(9.5,5.598076)(8,3)
						\Line(9.5,5.598076)(9.5,0.401924)
						\Line(8,3)(9.5,0.401924)
						
						{\color{red}\linethickness{1.1pt}
							\put(12.5,0.401924){\vector(0,1){4.806152}}
							\put(14,3){\vector(-1.732051,1){4.16225}}
						}
						
						\linethickness{0.4pt}
						\put(12.5,0.401924){\color{white}\circle*{0.78}}
						\put(12.5,0.401924){\color{black}\circle{0.78}}
						\put(12.5,0.401924){\makebox(0,0){\large $1$}}
						\put(14,3){\color{white}\circle*{0.78}}
						\put(14,3){\color{black}\circle{0.78}}
						\put(14,3){\makebox(0,0){\large $2$}}
						\put(12.5,5.598076){\color{white}\circle*{0.78}}
						\put(12.5,5.598076){\color{black}\circle{0.78}}
						\put(12.5,5.598076){\makebox(0,0){\large $3$}}
						\put(9.5,5.598076){\color{white}\circle*{0.78}}
						\put(9.5,5.598076){\color{black}\circle{0.78}}
						\put(9.5,5.598076){\makebox(0,0){\large $4$}}
						\put(8,3){\color{white}\circle*{0.78}}
						\put(8,3){\color{black}\circle{0.78}}
						\put(8,3){\makebox(0,0){\large $5$}}
						\put(9.5,0.401924){\color{white}\circle*{0.78}}
						\put(9.5,0.401924){\color{black}\circle{0.78}}
						\put(9.5,0.401924){\makebox(0,0){\large $6$}}
					\end{picture}%
				}
				
			\end{picture}
		\end{center}
		\endgroup
		\caption{A tuple $\bnu$ of columns (left), the corresponding Schr\"oder path $p=p(\bnu)=\mathrm{NNDDNNEEEE}$ (middle) and the decorated unit-interval graph $\Gamma_p$ (right).}\label{fig:path1}
	\end{figure}
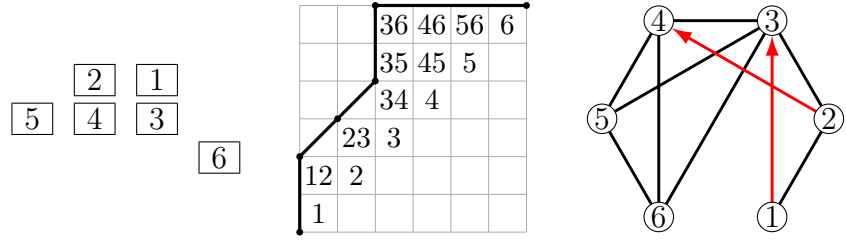
	
	\subsection{Dyck paths and unicellular LLT polynomials}
	A {\em \color{blue} unicellular LLT polynomial} is a vertical-strip LLT polynomial $\G_{\bnu}(X;t)$ indexed by a tuple $\bnu$ of single cells; namely, $\G_{\bnu}(X;t)=G_p(X;t)$ for some Dyck path $p$ by Proposition \ref{prop:path-vertical}. Griffin, Mellit, Romero, Weigl and Wen \cite[Theorem 4.1]{GMRWW:25}, and, independently,  Basu and Bhattacharya \cite[Proposition 5.1]{BB}, derived a combinatorial formula for unicellular LLT polynomials in the $q$-Whittaker basis.
	
	We state this formula after introducing the necessary statistics on the decorated unit-interval graphs and Schr\"oder paths.
	\begin{definition}[{\em \color{blue}Statistics on paths and tableaux}]\label{def}
		Let $\lambda$ be a partition of $n$. Let $(p,T)$ be a pair, where $T$ is a standard Young tableau of shape $\lambda$  and $p$ is  a Schr\"oder path of length $n$. Let $(\row_T(j),\col_T(j))$ denote the box filled with $j$ in $T$, and let $\Gamma_p=([n],E,S)$ be the decorated unit-interval graph corresponding to $p$.
		Define $\SYT(p;\lambda)$ to be the set of standard Young tableaux $T$ of shape $\lambda$ satisfying both of the following conditions:
		\begin{enumerate}
			\item If $ij\in E-S$ with $i<j$, then $\col_T(i)\neq \col_T(j)$.
			\item If $ij\in S$ with $i<j$, then $\row_T(i)<\row_T(j)$.
		\end{enumerate}
		Define $\L(T)=\{j\in [n]\,\vert\, \row_T(j)>1\}$, i.e.,   the set of entries not in the top row,  and
		\[
		\inv(p,T)=\#\{ij\in E-S: \row_T(i)< \row_T(j)\,\text{ and }\, i<j\}.
		\]
		For each $j\in \L(T)$, we define
		\begin{equation*}
			\arm_{j}(p,T)=\#\left\{
			i<j\,\vert\,
			\row_T(i)=\row_T(j)-1, \, \col_T(i)>\col_T(j), \, ij\notin E-S
			\right\}.
			\label{eq:k}
		\end{equation*}
		Finally, we define $\gamma(p,T)=\{j\in[n]\,\vert\,ij\in S,\ \col_T(i)=\col_T(j)\}$.
	\end{definition}

	\begin{theorem}\cite[Proposition 5.1]{BB}\cite[Theorem 4.1]{GMRWW:25} \label{thm:unicellular}
		For a Dyck path $p$ of length $n$ and a partition $\lambda$ of $n$, we have
		\begin{equation}
			\langle W_\lambda(X;t)\rangle G_p(X;t)=
			\sum_{T\in\SYT(p;\lambda)}
			t^{\inv(p,T)}\prod_{j\in \L(T)}(1-t^{\arm_{j}(p,T)+1}).
			\label{eq:unicellular}
		\end{equation}
	\end{theorem}

	\section{Proof of Theorem \ref{T:main}, Part I}\label{S:3}
	The purpose of this section is to establish (\ref{E:productEP}) of Theorem \ref{T:main}. 
	Before we proceed, we need to introduce some notation and a preparatory lemma.
	
	Let $\tau:\mathbb{Q}[x_1^{\pm 1},\ldots,x_n^{\pm 1}]\rightarrow \mathbb{Q}[x_1^{\pm 1},\ldots,x_n^{\pm 1}]$
	be the {\em \color{blue} reciprocal-reversal involution} defined by $$\tau f(x_1,\ldots,x_n)=f(x_n^{-1},\ldots,x_1^{-1}).$$ Noting that the proof is entirely analogous to the one for Lemma \ref{E:etof}, we omit the details.
	\begin{lemma}\label{lem:tau}
		For every $\sigma\in \mathbb{Z}^n$, we have 
		\begin{align*}
			\tau F_{\sigma}(t)=F_{-w_0\sigma}(t)=F_{-\sigma^r}(t).
		\end{align*}
	\end{lemma}
	We now turn to one-row diagrams. For integers $a\geq b$, let $R=a/b=\{(1,j):j\in\mathbb Z,\ b<j\leq a\}$.
		Thus $R$ is a row of consecutive boxes with contents $b,b+1,\ldots,a-1$.
		If $a=b$, then $R=a/a=\emptyset$.
		A {\em \color{blue} horizontal-strip tuple}
		$\bnu=(R_1,\ldots,R_m)$ is a tuple of such rows.
	
	\begin{definition}[{\em \color{blue} $h$-triple}]\label{Def:triples}
		Let $\bnu=(\nu^1,\ldots,\nu^l)$ be a tuple of skew diagrams. Adjoining the boxes to the left and right of a row gives an extended row. This includes the case of the two boxes immediately to  the left and right of an empty row $\emptyset$. A triple $(x,y,z)$ of $\bnu$ consists of three boxes $x$, $y$ and $z$ such that
		\begin{enumerate}
			\item $x$ and $z$ are adjacent boxes in an extended row of some component $\nu^j$, with $x$ left of $z$;
			\item $y$ is a box in $\nu^i$ for some $i<j$;
			\item the contents of these boxes satisfy $c(y)=c(z)$.
		\end{enumerate}
		The number of triples of $\bnu$ is denoted by $h(\bnu)$.
	\end{definition}
	
	\begin{proposition}\cite[Corollary 6.5.8]{BHMPPS:25}\label{cor:thmB1}
		For two weak compositions $\gamma,\eta\in\mathbb{N}^{n}$ with $\gamma\le \eta$, let  $\bnu=\eta/\gamma$ be a sequence of rows. Then, we have
		\begin{align*}
			\langle F_{\gamma}(Y_n;t^{-1})\rangle \Exp[X(y_1^{-1}+\cdots +y_{n}^{-1})]F_{\eta}(Y_n;t^{-1})=t^{-h({\bnu})}G_{\bnu}(X;t).
		\end{align*}
	\end{proposition}
	Proposition \ref{cor:thmB1} enables us to relate the coefficients $a_{\gamma\lambda}^{\eta}(t)$ defined in (\ref{E:defa}) to the expansion of vertical-strip LLT polynomials.
	\begin{theorem}\label{prop:LLTtoW}
		For two weak compositions $\gamma,\eta\in \mathbb{N}^n$ with $\gamma\le\eta$, define
		\begin{align*}
			\bnu=\bnu(\gamma,\eta)=((-\gamma_1)/(-\eta_1),\cdots, (-\gamma_n)/(-\eta_n)).
		\end{align*}
		Then, we have
		\begin{align}\label{E:aQ'=LLT}
			\sum_{\lambda}a_{\gamma\lambda}^{\eta}(t)Q'_{\lambda}(X_n;t)=t^{h(\bnu)}G_{\bnu}(X_n;t^{-1}),
		\end{align}
		where $t^{h(\bnu)}G_{\bnu}(X_n;t^{-1})$ is an $\mathbb{N}[t]$-polynomial.
		{Consequently,}
		\begin{align}\label{E:GW1}
			\sum_{\lambda}a_{\gamma\lambda}^{\eta}(t)W_{\lambda'}(X_n;t)=t^{h(\bnu)-m(\bnu)}G_{\bnu'}(X_n;t).
		\end{align}
	\end{theorem}
	\begin{proof}
		By Definition \ref{Def:Ea}, $K_{\gamma\eta}(X_n;t)=\langle\E_{\eta}(Y_n;t)\rangle (\Exp[X_nY_n]\E_{\gamma}(Y_n;t))$.
		We apply the reversal automorphism $w_0=n\,n-1\,\cdots 1$ to the variables $Y_n=(y_1,\ldots,y_n)$. Since $Y_n=y_1+\cdots +y_n$ is invariant under $w_0$, and by (\ref{E:EtoF}), we obtain
		\begin{align*}
			K_{\gamma\eta}(X_n;t)&=\langle\E_{\eta}(Y_n^r;t)\rangle (\Exp[X_nY_n]\E_{\gamma}(Y_n^r;t))\\
			&=\langle F_{w_0\eta}(Y_n;t)\rangle(\Exp[X_nY_n]F_{w_0\gamma}(Y_n;t))\\
			&=\langle F_{\eta^r}(Y_n;t)\rangle(\Exp[X_nY_n]F_{\gamma^r}(Y_n;t))
		\end{align*}
		By applying $\tau$ to the Laurent polynomials in the variable $Y_n$ on the right-hand side and using Lemma \ref{lem:tau}, we obtain
		\begin{align*}
			K_{\gamma\eta}(X_n;t)&=\langle F_{-\eta}(Y_n;t)\rangle\,\Exp[X(y_1^{-1}+\cdots +y_n^{-1})]F_{-\gamma}(Y_n;t)\\
			&=t^{h(\bnu)}G_{\bnu}(X_n;t^{-1})
		\end{align*}
		by Proposition \ref{cor:thmB1}. This yields (\ref{E:aQ'=LLT}) by (\ref{E:Ka}).
		Note that $G_{\bnu'}(X_n;t)$ is a vertical-strip LLT polynomial. We will show that $t^{h(\bnu)}G_{\bnu}(X_n;t^{-1})$ is an $\mathbb{N}[t]$-polynomial, i.e.,
		\begin{align*}
			h(\bnu)\ge \inv(\bT)
		\end{align*}
		for any $\bT\in \SSYT(\bnu)$. Write $\bnu=(R_1,\ldots,R_n)$, where $R_i=(-\gamma_i)/(-\eta_i)$ is a single row. Let $\bT=(T^1,\ldots,T^n)$ with $T^i\in\SSYT(R_i)$. For $i<j$, an attacking pair of $\bnu$ is either a pair $(y,z)$ such that $c(z)=c(y)$, $y\in R_i$ and $z\in R_j$, or a pair $(x,y)$ such that $c(y)=c(x)+1$, $y\in R_i$ and $x\in R_j$. For the former, we take the left neighbour of $z$, which we call $x$, to form a triple $(x,y,z)$. For the latter, we choose the right neighbour of $x$, denoting it by $z$, to form a triple $(x,y,z)$. The inversion pairs counted by $\inv(\bT)$ require that $T^i(y)>T^j(z)$ in the former case, or $T^j(x)>T^i(y)$ in the latter. Since the entries along each row of $T^j$ are weakly increasing from the left, we have $T^j(x)\le T^j(z)$, implying that  $T^i(y)>T^j(z)$ and $T^j(x)>T^i(y)$ cannot hold simultaneously. Consequently, $\inv(\bT)\le h(\bnu)$. It follows that $t^{h(\bnu)}G_{\bnu}(X_n;t^{-1})$ is an $\mathbb{N}[t]$-polynomial.
		
		We apply the involution $\omega$ to both sides of (\ref{E:aQ'=LLT}). By (\ref{E:omegaG}) and (\ref{E:qWhittaker}), we conclude (\ref{E:GW1}). {Since $\omega$ acts coefficientwise in $t$, the right-hand side of (\ref{E:GW1}) contains no negative powers of $t$.}
		{Nevertheless, $h(\bnu)-m(\bnu)$ is not always nonnegative.}   		
	\end{proof}
	
	By Proposition \ref{prop:LLTtoW}, the problem of establishing a combinatorial formula for $a_{\gamma\lambda}^{\eta}(t)$ is reduced to the study of the $q$-Whittaker expansion of vertical-strip LLT polynomials $G_p(X;t)$. We will deduce such a formula from Theorem \ref{thm:unicellular}. We write $\chi(A)=1$ if the statement $A$ is true; otherwise $\chi(A)=0$.
	\begin{theorem}\label{thm:positive}
		Given a Schr\"oder path $p$ of length $n$, let \(\Gamma_p = ([n], E, S)\) be the corresponding decorated unit-interval  graph. Then,
		\begin{equation}
			\langle W_{\lambda}(X;t)\rangle G_p(X;t)
			=\sum_{T\in\SYT(p;\lambda)}
			t^{\inv(p,T)}(1-t)^{\beta(p,T)}
			\prod_{j\in \L(T)-\alpha(p,T)}[\arm_{j}(p,T)+1]_t
			\label{eq:positive}
		\end{equation}
		where the statistics are given in Definition \ref{def} and
		\begin{align*}
			\alpha(p,T)&=\{j\in[n]\,\vert\,ij\in S,\ \row_T(i)=\row_T(j)-1\,\mbox{ and }\, \col_T(i)>\col_T(j)\},\\
			\beta(p,T)&=\#\L(T)-\#\gamma(p,T)-\#\alpha(p,T).
		\end{align*}
	\end{theorem}
	\begin{remark}
		Note that $\beta(p,T)\ge 0$ since $\gamma(p,T)$ and $\alpha(p,T)$ are two disjoint subsets of $\L(T)$.
	\end{remark}
	\begin{proof}
		Given a Schr\"oder path $p$, let $\Gamma_p=([n], E, S)$ be the corresponding decorated unit-interval graph. For a subset $B\subseteq S$, let $p_{B}$ be the Dyck path obtained from $p$ by replacing a diagonal step with $\mathrm{NE}$ if its edge belongs to $B$, and with $\mathrm{EN}$ otherwise. Then we claim
		\begin{equation}
			G_p(X;t)=(1-t)^{-|S|}\sum_{B\subseteq S}(-1)^{|B|}G_{p_B}(X;t).
			\label{eq:resolution}
		\end{equation}
		Note that the edge set of $\Gamma_{p_B}$ is exactly $(E-S)\cup B$ and all these edges are non-strict edges. For a coloring $\kappa$ on the vertices of $\Gamma_{p}=([n],E,S)$, the number of ascents of coloring $\kappa$ in $\Gamma_{p_B}=([n],(E-S)\cup B)$ is
		\begin{align*}
			\asc(\kappa)+\sum_{ij\in B}\chi(\kappa(i)<\kappa(j))
		\end{align*}
		by \eqref{eq:asc}. It follows that
		\begin{align*}
			\sum_{B\subseteq S}(-1)^{|B|}G_{p_B}(X;t)&=\sum_{\kappa:[n]\to\NN^+} t^{\asc(\kappa)}\prod_{i=1}^{n}x_{\kappa(i)} \sum_{B\subseteq S}\prod_{ij\in B}(-t^{\chi(\kappa(i)<\kappa(j))}) \\
			&=\sum_{\kappa:[n]\to\NN^+} t^{\asc(\kappa)}\prod_{i=1}^{n}x_{\kappa(i)} \prod_{ij\in S}(1-t^{\chi(\kappa(i)<\kappa(j))}) \\
			&=\sum_{\substack{\kappa:[n]\to\NN^+\\ \kappa(i)<\kappa(j)\ \text{for }ij\in S}}   t^{\asc(\kappa)} \prod_{i=1}^{n}x_{\kappa(i)} (1-t)^{|S|}=(1-t)^{|S|}G_p(X;t),
		\end{align*}
		which is equivalent to \eqref{eq:resolution}. Let $ \inv_j(p,T)=\#\{i<j\,\vert\,ij\in E-S\,\mbox{ and }\, \row_T(i)<\row_T(j)\}$, so that $\inv(p,T)=\sum_j \inv_j(p,T)$. By Definition \ref{def}, we find $\mathcal \SYT(p_{B};\lambda)\subseteq \SYT(p_{\emptyset};\lambda)$. By Theorem \ref{thm:unicellular}, we write
		\begin{align}\label{eq:1}
			\langle W_{\lambda}(X;t)\rangle G_{p_B}(X;t)
			&=\sum_{T\in\SYT(p_{B};\lambda)} \prod_{j\in \L(T)} t^{ \inv_j(p_B,T)} (1-t^{\arm_{j}(p_B,T)+1}) \notag\\
			&= \sum_{T\in\SYT(p_{\emptyset};\lambda)} \chi(T\in \SYT(p_{B};\lambda))\prod_{j\in \L(T)}  t^{ \inv_j(p_B,T)} (1-t^{\arm_{j}(p_B,T)+1}),
		\end{align}
		where,  if $T\notin \SYT(p_{B};\lambda)$, the  weight of the product on the right-hand side  is zero.
		For the Dyck path $p_{B}$,  let $E_{B}=(E-S)\cup B$. Then $E_{\emptyset}=E-S \subseteq  E_{B}  \subseteq E$.
		Define $$H(S)=\{j\in[n] \,\vert\, \text{ there exists an }\, i_j\, \,\text{ with } \,i_j<j \text{ and }\, i_j\,j\in S  \}.$$
		Fix a tableau $T\in \SYT(p_{\emptyset};\lambda)$. For each $B\subseteq S$, we define the tableau contribution from $T$ by
		\begin{align*}
			w_B(T)=\chi(T\in \SYT(p_{B};\lambda))t^{\inv(p_B,T)}\prod_{j\in \L(T)}(1-t^{\arm_j(p_B,T)+1}).
		\end{align*}
		Therefore, by (\ref{eq:1}), the contribution of $T$ to $\langle W_{\lambda}(X;t)\rangle \sum_{B\subseteq S}(-1)^{|B|}G_{p_B}(X;t)$ equals
		\begin{align}\label{E:wB}
			\sum_{B\subseteq S}(-1)^{|B|}w_B(T).
		\end{align}
		Since distinct diagonal steps have distinct endpoint heights $j$, there is a unique edge $i_j\,j\in S$ for every $j\in H(S)$. Thus $S$ is indexed by $H(S)$, and choosing $B\subseteq S$ is equivalent to choosing $\chi(i_j\,j\in B)$ for $j\in H(S)$. We first examine the contributions to (\ref{E:wB}) when $i_j\,j\in B$. Define
		\begin{align*}
			\alpha_j(T)&=\chi( \row_T(i_j) <\row_T(j) ), \\
			\beta_j(T)&= \chi( \row_T(i_j) = \row_T(j) - 1,\; \col_T(i_j) > \col_T(j) ),\\
			\delta_j(T)&=\chi( \col_T(i_j) \ne \col_T(j) )
		\end{align*}
		We now discuss the change of statistics. If $j\in H(S)$, then
		\begin{align}\label{E:invarm1}
			\inv_j(p_{B},T)&=\inv_j(p,T)+  \chi( i_j\,j\in B \text{ and } \row_{T}(i_j)<\row_{T}(j) ) \\
			\arm_j(p,T)&=\arm_j(p_{B},T)+ \chi( i_j\,j\in B, \row_{T}(i_j)=\row_{T}(j)-1 \text{ and } \col_{T}(i_j)>\col_{T}(j)   )\notag.
		\end{align}
		Otherwise $j\notin H(S)$, and we have $\inv_j(p,T)=\inv_j(p_{B},T)$ and $\arm_j(p,T)=\arm_j(p_{B},T)$.
		If $i_j\,j\in B$, then (\ref{E:invarm1}) implies that
		\begin{align*}
			\inv_j(p_B,T)&=\inv_j(p,T)+\alpha_j(T),\\
			\arm_j(p_B,T)&=\arm_j(p,T)-\beta_j(T).
		\end{align*}
		Moreover, $T\in \SYT(p_B;\lambda)$ if and only if $\delta_j=1$ {for every $j\in H(S)$ with $i_jj\in B$}. Consequently, we separate the common factor $R(p,T)$ in $w_B(T)$ from the factors depending on $B$, which is defined by
		\begin{align*}
			R(p,T)=t^{\inv(p,T)}\prod_{j\in \L(T)-H(S)}(1-t^{\arm_j(p,T)+1}).
		\end{align*}
		For each $j\in H(S)$, we set
		\begin{align*}
			F_j(1) &= \delta_j(T) t^{\alpha_j(T)}
			\begin{cases}
				1-t^{\arm_j(p,T)+1-\beta_j(T)}, & j \in \L(T), \\[4pt]
				1, & j \notin \L(T).
			\end{cases}  \\
			F_j(0) &=
			\begin{cases}
				1-t^{\arm_j(p,T)+1}, & j \in \L(T), \\[4pt]
				1, & j \notin \L(T). \\
			\end{cases}
		\end{align*}
		These are exactly the remaining factors in $w_B(T)$; that is, let $\varepsilon_j=\chi(i_j\,j\in B)$, and
		\begin{align*}
			w_B(T)=R(p,T)\prod_{j\in H(S)}F_j(\varepsilon_j).
		\end{align*}
		Returning to (\ref{E:wB}), we find
		\begin{align*}
			\sum_{B\subseteq S}(-1)^{|B|}w_B(T)&=R(p,T)\sum_{\varepsilon_j\in \{0,1\}^{H(S)}}\prod_{j\in H(S)}(-1)^{\varepsilon_j}F_j(\varepsilon_j)
			=R(p,T)\prod_{j\in H(S)}(F_j(0)-F_j(1)).
		\end{align*}
		It follows from \eqref{eq:resolution} that
		\begin{align*}
			\langle W_{\lambda}(X;t)\rangle G_p(X;t)=\sum_{T\in\SYT(p_{\emptyset};\lambda)}R(p,T)\prod_{j\in H(S)} \frac{F_j(0)-F_j(1)}{1-t}.
		\end{align*}
		We will analyze the difference $(1-t)^{-1} (F_j(0)-F_j(1))$ by distinguishing the following four cases. For each $j\in H(S)$, we fix $i_j\,j\in S$.
		\begin{enumerate}[label=\arabic*:]
			\item If $\row_T(i_j)\geq \row_T(j)$, then $\alpha_j(T)=\beta_j(T)=0$, and $\col_T(i_j) \neq \col_T(j)$ since $T$ is a standard tableau. Thus $F_j(0)-F_j(1)=0$.
			\item If $j\in \gamma(p,T)$, i.e., $\col_T(i_j)=\col_T(j)$, then $i_j$ lies above $j$ in $T$; thus $j\in \L(T)$ and $F_j(1)=0$. Therefore,
			\begin{align*}
				\frac{F_j(0)-F_j(1)}{1-t}= \frac{1-t^{\arm_{j}(p,T)+1}}{1-t}=[\arm_{j}(p,T)+1]_t.
			\end{align*}
			\item If $j\in \alpha(p,T)$, that is, $\row_T(i_j)=\row_T(j)-1$ and $\col_T(i_j)>\col_T(j)$, then $j\in \L(T)$,  $\alpha_j(T)=\beta_j(T)=1$, and
			\begin{align*}
				\frac{F_j(0)-F_j(1)}{1-t}&= \frac{1-t^{\arm_{j}(p,T)+1} -t(1-t^{\arm_{j}(p,T)})}{1-t}=1.
			\end{align*}
			\item If $\row_{T}(i_j)<\row_T(j)$ and $j\not\in \gamma(p,T)\,\dot\cup\,\alpha(p,T)$, then $j\in \L(T)$, $\alpha_j(T)=1, \beta_j(T)=0$, and we have
			\begin{align*}
				\frac{F_j(0)-F_j(1)}{1-t}= \frac{1-t^{\arm_{j}(p,T)+1}-t(1-t^{\arm_{j}(p,T)+1})}{1-t}
				=1-t^{\arm_{j}(p,T)+1}.
			\end{align*}
		\end{enumerate}
		{By Case~1, any tableau $T$ for which $\row_T(i_j)\geq\row_T(j)$ for some $j\in H(S)$ contributes zero.
			We therefore restrict to tableaux satisfying $\row_T(i_j)<\row_T(j)$ for every $j\in H(S)$. In particular, $H(S)\subseteq\L(T)$.}
		Let $D=H(S)-(\gamma(p,T)\,\dot\cup \,\alpha(p,T))$, then $\L(T)=(\L(T)-H(S))\,\dot\cup\,\alpha(p,T)\,\dot\cup\,\gamma(p,T)\,\dot\cup\,D$. This analysis of the four cases yields
		\begin{align*}
			R(p,T)\prod_{j\in H(S)} \frac{F_j(0)-F_j(1)}{1-t}&=t^{\inv(p,T)}\prod_{j\in \L(T)-H(S)}(1-t)[\arm_j(p,T)+1]_t\\
			&\qquad\times \prod_{j\in \gamma(p,T)}[\arm_j(p,T)+1]_t\times \prod_{j\in D}(1-t)[\arm_j(p,T)+1]_t.
		\end{align*}
		The number of factors $1-t$ is
		$$\#(\L(T)-H(S))+\# D=\#\L(T)-\#\gamma(p,T)-\#\alpha(p,T)=\beta(p,T),$$
		and the arm factors occur  precisely at the entries in $$\L(T)-H(S)+\gamma(p,T)+D=\L(T)-\alpha(p,T).$$
		Moreover, the set $\SYT(p_{\emptyset};\lambda)$ imposes the non-strict edge column conditions, that is condition (1) of Definition \ref{def}. Requiring every strict edge in $S$ to avoid Case 1 adds exactly condition (2) of Definition \ref{def}. Consequently, the sum ranges over $T\in \SYT(p;\lambda)$. This completes the proof of (\ref{eq:positive}).
	\end{proof}
	
	For two weak compositions $\gamma,\eta\in \mathbb{N}^n$ with $\gamma\le \eta$, recall that $\bnu=\bnu(\gamma,\eta)=(-\gamma)/(-\eta)=((-\gamma_1)/(-\eta_1),\ldots, (-\gamma_n)/(-\eta_n))$ is a horizontal strip. Then $\bnu'=\bnu'(\gamma,\eta)$ is a vertical strip. See Example \ref{Example:bnu'}.
	\begin{example}\label{Example:bnu'}
		Let $\gamma=(0,1,1,1)$ and $\eta=(1,3,3,2)$ with $\gamma\le \eta$; then the horizontal strip $\bnu=(-\gamma)/(-\eta)$ is drawn on the left, which, after transposing each row and reversing the order of the columns, gives the vertical-strip $\bnu'=( (1,1)/(1), (1,1),  (1,1), (1,1,1)/(1,1))$ on the right.
		\ytableausetup{boxsize=1.5em,centertableaux}
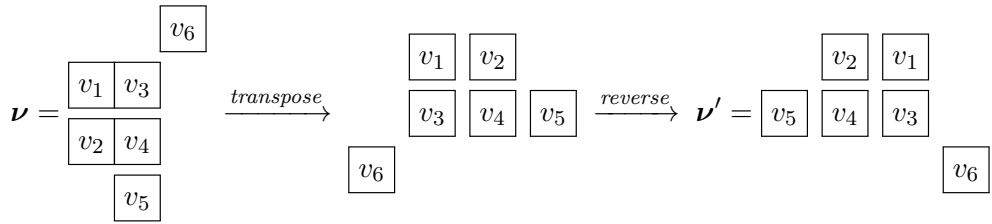
		\[
		\boldsymbol{\nu}=
		\begin{ytableau}
			\none & \none & v_6  \\
			\noalign{\vskip4pt}
			v_1 & v_3 & \none \\
			\noalign{\vskip4pt}
			v_2 & v_4 & \none  \\
			\noalign{\vskip4pt}
			\none & v_5 & \none
		\end{ytableau}
		\;\xrightarrow[]{\text{transpose}}\;
		\begin{ytableau}
			\none & \omit\hskip6pt & v_1 & \omit\hskip6pt & v_2 & \omit\hskip6pt & \none \\
			\noalign{\vskip4pt}
			\none & \omit\hskip6pt & v_3 & \omit\hskip6pt & v_4 & \omit\hskip6pt & v_5 \\
			\noalign{\vskip4pt}
			v_6 & \omit\hskip6pt & \none & \omit\hskip6pt & \none & \omit\hskip6pt & \none
		\end{ytableau}
		\;\xrightarrow[]{\text{reverse}}\;
		\boldsymbol{\nu}'=
		\begin{ytableau}
			\none & \omit\hskip6pt & v_2 & \omit\hskip6pt & v_1 & \omit\hskip6pt & \none \\
			\noalign{\vskip4pt}
			v_5 & \omit\hskip6pt & v_4 & \omit\hskip6pt & v_3 & \omit\hskip6pt & \none \\
			\noalign{\vskip4pt}
			\none & \omit\hskip6pt & \none & \omit\hskip6pt & \none & \omit\hskip6pt & v_6
		\end{ytableau}
		\]
	\end{example}

	Proposition \ref{prop:path-vertical} states that there exists a Schr\"oder path $p=p(\bnu')$ such that $\G_{\bnu'}(X;t)=G_p(X;t)$. Together with \eqref{E:GW1} and (\ref{eq:positive}), we conclude
	\begin{align}\label{E:a_path}
		a_{\gamma\lambda}^{\eta}(t)&=t^{h(\bnu)-m(\bnu)}\langle W_{\lambda'}(X;t)\rangle G_{p}(X;t)\notag\\
		&=t^{h(\bnu)-m(\bnu)}\sum_{T\in\SYT(p;\lambda')}
		t^{\inv(p,T)}(1-t)^{\beta(p,T)}
		\prod_{j\in \L(T)-\alpha(p,T)}[\arm_{j}(p,T)+1]_t.
	\end{align}
	We now translate all the statistics into tableau-theoretic terms in the skyline diagram $\dg(\eta/\gamma)$ to derive the formula (\ref{E:productEP}). This will be useful for connecting its specialization at $t=0$ with the Littlewood-Richardson skew skyline tableaux in Theorem \ref{T:Hthm4.1}.
	{In order to avoid confusion between the two kinds of statistics, we mark the tableau version with a star $*$.}

	\begin{definition}[{\em \color{blue} Attacking pair in the skyline diagram}]\label{def:inv-attack}
		Let $\eta,\gamma\in\mathbb{N}^n$ be two weak compositions satisfying $\gamma\leq\eta$, and set $N=|\eta|-|\gamma|$.
		Two ordinary boxes $u,v$ of $\dg(\eta/\gamma)$ are said to be
		{\color{blue} attacking} if either $u=(i,k),v=(j,k)$,
		or $u=(j,k),v=(i,k-1)$, for some $i<j$. Let $m(\gamma,\eta)$ denote the number of attacking pairs in $\dg(\eta/\gamma)$.
		The relative positions of the boxes are illustrated below.
		
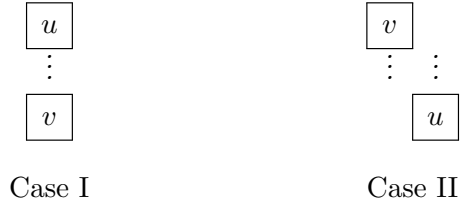
\begin{figure}[H]
			\centering
			\begin{minipage}[H]{0.3\linewidth}
				\centering
				\begin{ytableau}
					u\\
					\none[\vdots]\\
					v
				\end{ytableau}
				\caption*{Case I}
			\end{minipage}
			\begin{minipage}[H]{0.3\linewidth}
				\centering
				\begin{ytableau}
					v & \none\\
					\none[\vdots] & \none[\vdots]\\
					\none & u
				\end{ytableau}
				\caption*{Case II}
			\end{minipage}
		\end{figure}
		Define
		$
		H(\gamma,\eta)=\bigl\{ (i,j,k): 1\leq i<j\leq n,\ \gamma_i<k\leq\eta_i,\ \gamma_j\leq k\leq\eta_j \bigr\}$
		and let
		$h(\gamma,\eta)=\#H(\gamma,\eta)$.
		In other words, $h(\gamma,\eta)$ equals the number of pairs $(u,v)$ of boxes
		$u=(i,k)$ and $v=(j,k)$ in the same column, with $i<j$, such that $u$ is ordinary and $v$ is either ordinary or the rightmost basement box in row $j$.
		
		Label the ordinary boxes of $\dg(\eta/\gamma)$ as
		$v_1,\ldots,v_N$ in column-reading order: read the columns from right to left and the boxes within
		each column from top to bottom. 
	\end{definition}
	\begin{definition}[Reinterpretation of Definition \ref{def} by Proposition \ref{prop:path-vertical}]\label{Def:etagamma}
		For two weak compositions $\gamma,\eta\in \mathbb{N}^n$ with $\gamma\leq\eta$, 
		recall that $v_i$ denotes the $i$th ordinary box of
		$\dg(\eta/\gamma)$ in column-reading order. 

		Let $\SYT(\lambda;\eta,\gamma)$ denote the set of standard Young tableaux $T$ of shape $\lambda$ with the following two properties:
		\begin{enumerate}
			\item If $(v_i,v_j)$ is attacking in $\dg(\eta/\gamma)$ with $i<j$, then $\col_T(i)\neq \col_T(j)$.\\[-2mm]
			\item If $
			\raisebox{1pt}{\begin{ytableau}[] v_j&v_i \end{ytableau}}\in  \dg(\eta/\gamma)$ with $i<j$, then $\row_T(i)<\row_T(j)$.
		\end{enumerate}
		We now define
		\[
		\inv^*(T) =\#\left\{(i,j)\,\vert\,
		\text{$(v_i,v_j)$ is attacking in $\dg(\eta/\gamma)$ with $i<j$ and $\row_T(i)<\row_T(j)$}
		\right\}.
		\]
		Recall that $\L(T)=\{j\,\vert\, \row_T(j)>1\}$. For $j\in \L(T)$, let
		\begin{equation*}
			\arm_{j}^*(T)=\#\left\{i<j\,\bigg\vert\,
			\begin{array}{l}
				\row_T(i)=\row_T(j)-1,\,
				\col_T(i)>\col_T(j)\\[1mm]
				\text{and $(v_i,v_j)$ is not attacking in $\dg(\eta/\gamma)$}
			\end{array}
			\right\}.
		\end{equation*}
		Finally, we define
		\begin{align*}
			\gamma^*(T)&=\{j\,\big\vert\,
			\raisebox{1pt}{\begin{ytableau}[] v_j&v_i \end{ytableau}}\in  \dg(\eta/\gamma) \text{ and } \col_T(i)=\col_T(j)\},\\
			\alpha^*(T)&=\{j\,\big\vert\,
			\raisebox{1pt}{\begin{ytableau}[] v_j&v_i \end{ytableau}}\in  \dg(\eta/\gamma),\,\row_T(i)=\row_T(j)-1 \text{ and } \col_T(i)>\col_T(j)\},\\
			\beta^*(T)&=\# \L(T)-\# \gamma^*(T)-\#\alpha^*(T)=|\lambda|-\lambda_1-\# \gamma^*(T)-\#\alpha^*(T).
		\end{align*}
	\end{definition}
	
	We now establish the tableau formula (\ref{E:productEP}) for $a_{\gamma\lambda}^{\eta}(t)$ in Theorem \ref{T:main}.

	\begin{lemma}[Equation (\ref{E:productEP})]
		For weak compositions $\gamma,\eta\in \mathbb{N}^n$ with $\gamma\le \eta$ and a partition $\lambda$ of $|\eta|-|\gamma|$, we have
		\begin{align*}
			a_{\gamma\lambda}^{\eta}(t)&=
			\sum_{T\in\SYT(\lambda';\eta,\gamma)}
			t^{\inv^*(T)+h-m}(1-t)^{\beta^*(T)}
			\prod_{j\in \L(T)- \alpha^*(T)}[\arm^*_{j}(T)+1]_t,
		\end{align*}
		where $h=h(\gamma,\eta)$ and $m=m(\gamma,\eta)$, as given in Definition \ref{def:inv-attack}. Consequently,
		\begin{equation}\label{eq:ddef}
			a_{\gamma\lambda}^{\eta}=a_{\gamma\lambda}^{\eta}(0)=\#\{ T\in\SYT(\lambda';\eta,\gamma)\,\vert\, \inv^*(T)+h=m\}.
		\end{equation}
	\end{lemma}
	\begin{proof}
		Combining (\ref{E:a_path}), Proposition \ref{prop:path-vertical} and the notation
		of Definition~\ref{Def:etagamma}, we obtain
		\begin{equation*}
			a_{\gamma\lambda}^{\eta}(t)=t^{h(\bnu)-m(\bnu)}
			\sum_{T\in\SYT(\lambda';\eta,\gamma)}
			t^{\inv^*(T)}(1-t)^{\beta^*(T)}
			\prod_{j\in \L(T)-\alpha^*(T)}[\arm_{j}^*(T)+1]_t,
		\end{equation*}
		where $\bnu=(-\gamma)/(-\eta)$.
		It remains to prove that $h(\bnu)=h(\gamma,\eta)$ and $m(\bnu)=m(\gamma,\eta)$.
		By Definitions \ref{Def:triples}  and \ref{def:inv-attack}, the map $(x,y,z)\mapsto(y,z)$ is indeed a bijection between the $h$-triples and the pairs of boxes counted by $h(\gamma,\eta)$. Hence $h(\bnu)=h(\gamma,\eta)$. Similarly, Definition \ref{Def:attack} gives  $m(\bnu)=m(\gamma,\eta)$. This proves (\ref{E:productEP}).
		
	\end{proof}
	\begin{example}\label{exam2}
		Let $\eta=(4,4,2)$ and $\gamma=(1,1,0)$.
		The diagram $\dg(\eta/\gamma)$ is shown below on the left, with each
		ordinary box labeled by its position in column-reading order:
		\[
		\vcenter{\hbox{\begin{ytableau}
					*(green) * & \ast & v_5 & v_3 & v_1 \\
					*(green) * & \ast & v_6 & v_4 & v_2 \\
					*(green) * & v_8  & v_7
		\end{ytableau}}}
		\qquad\quad
		\vcenter{\hbox{$
				T=\begin{array}{ccc}
					1&2&7\\
					3&4&\\
					5&6&\\
					8&&
				\end{array}$}}
		\]
		There are $7$ attacking pairs in $\dg(\eta/\gamma)$ by Definition \ref{def:inv-attack}, $5$ of which lie in a single column:
		\begin{align*}
			& \{(v_1,v_2),(v_2,v_3),(v_3,v_4),(v_4,v_5),(v_5,v_6),(v_5,v_7),(v_6,v_7)\}, \\
			& \{(v_5,v_6),(v_3,v_4),(v_1,v_2),(v_5,v_7),(v_6,v_7)\},
		\end{align*}
		Thus, $m(\gamma,\eta)=7$ and $h(\gamma,\eta)=5$.
		Let $\lambda'=(3,2,2,1)$. Then the tableau $T$ shown above on the right is a standard Young tableau, $T\in \SYT(\lambda';\eta,\gamma)$ by Definition \ref{Def:etagamma}.
		
		Among the attacking pairs $(v_i,v_j)$ with $i<j$, the only ones
		satisfying $\row_T(i)<\row_T(j)$ are $(v_2,v_3)$ and $(v_4,v_5)$, so that $\inv^*(T)=2$. It follows that
		$\inv^*(T)+h(\gamma,\eta)-m(\gamma,\eta)=0$.
		Moreover, $\gamma^*(T)=\{3,4,5,6\}$, $\alpha^*(T)=\emptyset$ and $\beta^*(T)=1$.	    	
	\end{example}

		\section{Proof of Theorem \ref{T:main}, Part II}\label{S:4}
		This section is devoted to showing that, at $t=0$, Theorem~\ref{T:main} reduces to Theorem \ref{T:Hthm4.1}.
		We now introduce the Littlewood-Richardson skew skyline (LRS) tableaux of \cite[Sections  3 and 4]{HLMvW10}, which yield the combinatorial formula for $a_{\gamma\lambda}^{\eta}$ in Theorem \ref{T:Hthm4.1}.
		
		A \emph{\color{blue}skyline filling} $\sigma$ of $\dg(\eta/\gamma)$ is an assignment of a positive integer to each ordinary box of $\dg(\eta/\gamma)$, and we require the entries in each row
		to be weakly decreasing from left to right.
		For the purpose of comparing entries, we regard each basement symbol $\ast$ as a formal entry larger than every positive integer.
		Basement symbols in the same row are considered equal, whereas, among basement symbols in the same column, the one in the higher row is considered larger.
		\begin{definition}[{\em \color{blue} Inversion triple and column word}]
			For a skew filling $\sigma$ of the diagram $\dg(\eta/\gamma)$, fix rows $i<j$.
			If $\eta_i\geq\eta_j$, then a {triple of type $A$} is a
			triple $(a,b,c)$ of boxes in $\dg(\eta/\gamma)$ with $a=\sigma(i,k)$, $b=\sigma(j,k)$ and $c=\sigma(i,k-1)$.
			If $\eta_i<\eta_j$, then a triple of type $B$ is a triple $(a,b,c)$ of boxes with
			$a=\sigma(j,k+1)$, $b=\sigma(i,k)$ and $c=\sigma(j,k)$.
			
			The two types of triples are illustrated in Figure \ref{fig:typeAB}. The boxes $a,b,c$ need not be ordinary boxes: any of them may belong to the extended basement and hence contain $\ast$. Since the entries in each row of $\sigma$ are weakly decreasing, we have $a\leq c$.
			A triple $(a,b,c)$ is called an {\color{blue} inversion triple} if
			\begin{equation}
				b<a\leq c
				\qquad\text{or}\qquad
				a\leq c<b.
				\label{eq:inversion-triple}
			\end{equation}
			In the remaining case, $a\leq b\leq c$, it is called a {\color{blue}coinversion triple}.
			The {\color{blue} column word} of a filling $\sigma$, denoted by $\operatorname{col}(\sigma)$, is obtained by reading the
			ordinary entries column by column from right to left, and from top to bottom within each column. All basement symbols $\ast$ are omitted.
		\end{definition}
		
		\begin{figure}[H]
			\centering
			\begin{minipage}[H]{0.25\linewidth}
				\centering
				\begin{ytableau}
					c & a\\
					\none & \none[\vdots]\\
					\none & b
				\end{ytableau}
				\caption*{type $A$}
			\end{minipage}
			\begin{minipage}[H]{0.25\linewidth}
				\centering
				\begin{ytableau}
					b & \none\\
					\none[\vdots] & \none\\
					c & a
				\end{ytableau}
				\caption*{type $B$}
			\end{minipage}
			\begin{minipage}[H]{0.4\linewidth}
				\centering
				\[\sigma=\begin{ytableau}
					*(green) * & \ast & 1 \\
					*(green) * & 2 \\
					*(green) * &\ast &  \ast & 2 \\
					*(green) * & 1 \\
					*(green) * & \ast & \ast & 3 & 3
				\end{ytableau}
				\]
			\end{minipage}
			\caption{Triples of two types and a Littlewood-Richardson skew skyline tableau (rightmost)}
			\label{fig:typeAB}
		\end{figure}
		\begin{definition}[{\em \color{blue} Regular contre-lattice}]
			A word $w$ over the positive integers is called {\color{blue} contre-lattice} if, in every prefix of $w$, the number of occurrences of $j$ is at least the number of occurrences of $j-1$ for every $2\leq j\leq\max(w)$.
			Such a word is called {\color{blue} regular} if its smallest letter is $1$. By convention, the empty word is regular.
		\end{definition}
		\begin{definition}[{\em \color{blue} Littlewood-Richardson skew skyline (LRS) tableaux}]\label{Def:LRS}
			A skyline filling $\sigma$ is called a {Littlewood-Richardson skew skyline (LRS) tableau} if $\operatorname{col}(\sigma)$ is a regular contre-lattice word
			and every type $A$ and type $B$ triple is an inversion triple.
			
			The {\color{blue} (filling) type} of a tableau $\sigma$ is the sequence $(c_1,c_2,\ldots,c_N)$ where $N$ is the largest entry of $\sigma$ and $c_i$ is the number of occurrences of $i$ in $\sigma$.
		\end{definition}
		\begin{example}
			Let $\eta=(2,1,3,1,4), \gamma=(1,0,2,0,2)$. Then the tableau $\sigma$ in Figure \ref{fig:typeAB}, with column word $\col(\sigma)=323121$, is a Littlewood-Richardson skew skyline (LRS) tableau of shape $\dg(\eta/\gamma)$.
		\end{example}
		
		\subsection{Proof of Theorem \ref{T:main} at $t=0$}
		By \eqref{eq:ddef}, it suffices to prove a bijection
		\begin{align}\label{eq:set}
			\psi:\,	&\bigl\{
			T\in\SYT(\lambda';\eta,\gamma)\,\vert\,
			\inv^*(T)
			+h(\gamma,\eta)=m(\gamma,\eta)
			\bigr\}
			\notag\\
			&\quad\rightarrow
			\{\text{LRS tableaux of shape }
			\dg(\eta/\gamma)
			\text{ with type } \lambda^r\}
		\end{align}
		The construction of $\psi$ relies on the following lemma.
		\begin{lemma}\label{L:pq}
			Let $T\in\SYT(\lambda';\eta,\gamma)$ and let $p<q$.
			Suppose that the boxes $(v_p,v_q)$ of ordinary boxes in $\dg(\eta/\gamma)$
			are either attacking or such that $v_p$ is immediately to the right of $v_q$. If $\row_T(p)<\row_T(q)$, then $\col_T(p)\geq \col_T(q)$.
		\end{lemma}
		\begin{proof}
			Suppose, to the contrary, that $\row_T(p)<\row_T(q)$ and $\col_T(p)<\col_T(q)$.
			Let $j$ be the entry in the box of $T$ at coordinates $(\row_T(q),\col_T(p))$. Since the entries in each row (or column) of $T$ are strictly increasing from the left (or from the top), we have $p<j<q$.

			{Since $(v_p,v_q)$ is either an attacking pair or a pair such that $v_p$ is immediately to the right of $v_q$, it follows from the  column-reading order of $\dg(\eta/\gamma)$ that $(v_p,v_j)$ is attacking. The defining conditions of $\SYT(\lambda';\eta,\gamma)$
				therefore require $\col_T(p)\neq \col_T(j)$.}
			However, $\col_T(p)=\col_T(j)$, a contradiction.
		\end{proof}

		\noindent\emph{Proof of \eqref{eq:set}.}
		Let $\lambda=(\lambda_1\geq\cdots\geq\lambda_l>0)$ and
		$T\in\SYT(\lambda';\eta,\gamma)$.
		Write $N=|\lambda|=|\eta|-|\gamma|$ and, for $1\leq k\leq N$, define
		\begin{equation}\label{eq:map}
			w_k=l+1-\col_T(k).
		\end{equation}
		Fill each ordinary box $v_k$ of $\dg(\eta/\gamma)$
		with entry $w_k$, and denote the resulting filling by $\psi(T)$. For simplicity, write $\sigma=\psi(T)$; we will verify that {whenever $\inv^*(T)
			+h(\gamma,\eta)=m(\gamma,\eta)$,  $\sigma$ is an LRS tableau with filling type $\lambda^r=(0<\lambda_l\le\ldots\le\lambda_1)$.}
		
		By construction, for each $1\leq i\leq l$, exactly $\lambda_i$ boxes of $\sigma$ contain $l+1-i$. Thus, $\sigma$ has type $\lambda^r$. For each $1\leq p\leq N$, the boxes of $T$ containing the integers $1$
		through $p$ form a standard tableau in which the column heights are weakly decreasing from left to right. Consequently, $\operatorname{col}(\sigma)$ is a regular
		contre-lattice word.
		
		By \eqref{eq:map} and Definition~\ref{Def:etagamma} (1), we have $\sigma(v_i)=w_i\neq w_j=\sigma(v_j)$ provided that $(v_i,v_j)$ is attacking in $\dg(\eta/\gamma)$.
		Note that \eqref{eq:map} also gives
		\begin{align}\label{E:11}
			\col_T(p)\geq \col_T(q)
			\quad\Longleftrightarrow\quad
			\sigma(v_p)=w_p\leq w_q=\sigma(v_q).
		\end{align}
		Together with
		Definition~\ref{Def:etagamma} (2) and Lemma~\ref{L:pq}, it follows
		that each row of $\sigma$ is weakly decreasing from the left.
		
		It remains to verify that all type $A$ and type $B$ triples
		in $\sigma$ are inversion triples. Since $T$ is a standard Young tableau, the inequalities
		$p<q$ and $\col_T(p)>\col_T(q)$ imply $\row_T(p)<\row_T(q)$.
		By Lemma~\ref{L:pq}, Definition \ref{Def:etagamma} (1) and (\ref{E:11}), this yields
		\[
		\row_T(p)<\row_T(q)
		\quad\Longleftrightarrow\quad
		\col_T(p)>\col_T(q)
		\quad\Longleftrightarrow\quad
		\sigma(v_p)<\sigma(v_q)
		\]
		whenever $(v_p,v_q)$ is an attacking pair in
		$\dg(\eta/\gamma)$. Consequently, we find
		\[
		\begin{aligned}
			m(\gamma,\eta)-\inv^*(T)=
			\#\left\{
			(p,q)\,\vert\,
			(v_p,v_q) \text{ with $p<q$ is attacking and }\,\sigma(v_p)>\sigma(v_q)
			\right\}.
		\end{aligned}
		\]
		For each $(i,j,k)\in H(\gamma,\eta)$, let $a=\sigma(j,k+1)$, $b=\sigma(i,k)$ and $c=\sigma(j,k)$
		with the convention $\sigma(j,\eta_j+1)=0$. Then the triple of entries $(a,b,c)$ contributes to $m(\gamma,\eta)-\inv^*(T)$ only when
		\begin{equation}\label{eq:two-ineq}
			b>c  \quad\text{or}\quad  a>b.
		\end{equation}
		Since each row of $\sigma$ is weakly decreasing from the left,
		the two inequalities in \eqref{eq:two-ineq} cannot
		hold simultaneously.
		Moreover, neither inequality holds if and only if $a<b<c$.
		This implies
		\begin{equation}\label{eq:zero-stat}
			\begin{aligned}
				&\inv^*(T)+h-m
				=
				\#\bigl\{
				(i,j,k)\in H(\gamma,\eta)\,\vert\,
				\sigma(j,k+1)<\sigma(i,k)<\sigma(j,k)
				\bigr\},
			\end{aligned}
		\end{equation}
		where $h=h(\gamma,\eta)$ and $m=m(\gamma,\eta)$. In particular, we have $\inv^*(T)+h-m\in \mathbb{N}$ for each $T\in\SYT(\lambda';\eta,\gamma)$,
		and $a_{\gamma\lambda}^{\eta}$ also counts those tableaux for which the set on the right-hand side of \eqref{eq:zero-stat} is empty.
		
		{We now suppose that $\inv^*(T)+h-m=0$}. Fix two rows $i$ and $j$ with $i<j$. If $\eta_i<\eta_j$,
		consider a type $B$ triple $(a,b,c)$ in $\sigma$. Since $c\ge a$, \eqref{eq:zero-stat} and the attacking condition imply  $c\ge a>b$ or $b>c\ge a$. That is, it is an inversion triple.
		
		Now suppose that $\eta_i\ge\eta_j$ and that $\sigma$ contains a type $A$ coinversion triple
		$(a,b,c)$, so that $a<b<c$.
		Choose the largest $k$ such that $\sigma(i,k)<\sigma(j,k)$.
		By our convention, an upper basement entry is larger
		than a lower basement entry in the same column.
		Hence,  the box $(i,k)$ must be ordinary, so $k>\gamma_i$.
		If $k<\gamma_j$, then $\sigma(j,k+1)=\ast$, whereas
		the box $(i,k+1)$ exists and is ordinary since $\eta_i\geq\eta_j$.
		This contradicts the maximality of $k$.
		Therefore, $\gamma_j\leq k\leq\eta_j$. By the maximality of $k$, the attacking condition,
		and the weak decrease along each row, we obtain
		\[
		\sigma(j,k+1)< \sigma(i,k+1)\leq\sigma(i,k)< \sigma(j,k).
		\]
		Thus, $(i,j,k)\in H(\gamma,\eta)$ and $\sigma(j,k+1)<\sigma(i,k)<\sigma(j,k)$,
		contradicting the condition that \eqref{eq:zero-stat} equals zero.
		Consequently, every type $A$ and type $B$ triple in $\psi(T)$
		is an inversion triple, and $\psi(T)$ is an LRS tableau of shape $\dg(\eta/\gamma)$ with type $\lambda^r$.
		
		The map $\psi$ is clearly invertible: Read $w_1,\ldots,w_N$ from a skyline filling $\psi(T)$
		of $\dg(\eta/\gamma)$ in column reading order.
		Starting with an empty Young diagram, successively place
		$k$ at the bottom of column $l+1-w_k$, for $1\le k\le N$, thereby recovering $T$. This completes the bijective proof of (\ref{eq:set}).
		
		\qed


		\begin{example}
			For the tableau $T$ in Example \ref{exam2}, we have
			$\ell(\lambda)=3$, and the column indices of the entries
			from $1$ through $8$ are  $(1,2,1,2,1,2,3,1)$. Hence $(w_i)_{i=1}^8=(3,2,3,2,3,2,1,3)$, and
			\[\psi(T)=\begin{ytableau}
				*(green) * & \ast & 3 & 3& 3 \\
				*(green) * & \ast & 2 &2 &2 \\
				*(green) * & 3 & 1
			\end{ytableau}
			\]
			Conversely, reading $w=\operatorname{col}(\psi(T))$ and placing each label $k$ in column $4-w_k$, with entries increasing from top to bottom within each column, recovers $T$.
		\end{example}

		\section{Proof of Corollaries \ref{cor1} and \ref{cor21}}\label{S:5}
		Before proving Corollaries \ref{cor1} and \ref{cor21}, we recall some background on quasisymmetric functions.
		\begin{definition}[{\em \color{blue} Quasisymmetric function}]
			A {quasisymmetric function} is a formal power series of bounded degree $F\in \mathbb{Q}(t)[x_1,x_2,\ldots]$ such that for all $k$ and $i_1<i_2<\cdots<i_k$, we have
			\begin{align*}
				\langle x_{i_1}^{\alpha_1}x_{i_2}^{\alpha_2}\cdots x_{i_k}^{\alpha_k}\rangle F
				=\langle x_{1}^{\alpha_1}x_{2}^{\alpha_2}\cdots x_{k}^{\alpha_k}\rangle F
			\end{align*}
			for all compositions $(\alpha_1,\alpha_2,\ldots,\alpha_k)$.
		\end{definition}
		The polynomials $\CMcal{L}_{\alpha}(t)$ defined in (\ref{E:defL}) are quasisymmetric in $X_n=(x_1,\ldots,x_n)$, as proved in \cite[Proposition 7.1]{HLMvW11}. {For each $N$, the family $\{\L_{\alpha}(t)\,\vert\, \alpha\vDash N,\, \ell(\alpha)\le n\}$ forms a basis of the homogeneous quasisymmetric polynomials of degree $N$ in $X_n$.} 	
		The {\em \color{blue} quasisymmetric Schur function} is defined by $S_{\alpha}=S_{\alpha}(X_n)=\CMcal{L}_{\alpha}(0)$. For any weak composition $\eta\in \mathbb{N}^n$ such that $\eta^+=\beta$, Haglund, Luoto, Mason, and van Willigenburg
		\cite[Theorem~5.1, Proposition~3.4 and Remark~3.5]{HLMvW10} proved that
		\begin{align}\label{E:QSxS}
			\langle \CMcal{S}_{\beta}\rangle\CMcal{S}_{\alpha} s_{\lambda}=\sum_{\substack{\gamma^+=\alpha\\ \gamma\le \eta}}a_{\gamma\lambda}^{\eta},
		\end{align}
		which counts the Littlewood-Richardson skew composition (LRC) tableaux of shape $\beta/\alpha$ with filling type $\lambda^r$. These tableaux are defined as follows.
		\begin{definition}[{\em \color{blue} Littlewood-Richardson skew composition (LRC) tableaux}]
			Let $\sigma_1$ and $\sigma_2$ be LRS tableaux (in Definition \ref{Def:LRS}) of shapes
			$\eta/\gamma$ and $\rho/\tau$, respectively.
			We say that $\sigma_1$ and $\sigma_2$ are {\color{blue} equivalent} if
			\begin{enumerate}
				\item\(\sigma_1\) and \(\sigma_2\) have the same set of ordinary (non-$*$) entries in each column.
				\item \(\rho^+ = \eta^+\).
			\end{enumerate}
			A {\color{blue} Littlewood-Richardson skew composition (LRC) tableau}
			is an equivalence class of LRS tableaux.
			For an LRS filling $\sigma$ of shape $\eta/\gamma$, let $[\sigma]$ denote the equivalence class of $\sigma$, where the shape of $[\sigma]$ is defined to be $\eta^+/\gamma^+$, and the filling type of $[\sigma]$ is that of $\sigma$, which is independent of the representative.
		\end{definition}

		We now establish Corollary \ref{cor1} and (\ref{eq:keycoeff}) of Corollary \ref{cor21}.
		
		{\em Proof of Corollary \ref{cor1}}.
		Theorem ~\ref{T:main} implies that
		\begin{align*}
			\CMcal{L}_{\alpha}(t)P_{\lambda}(t)
			=\sum_{\gamma^+=\alpha}\E_{\gamma}(t)P_{\lambda}(t)
			=\sum_{\gamma^+=\alpha}\sum_{\eta\ge \gamma}  a_{\gamma\lambda}^{\eta}(t)\E_{\eta}(t)
			=\sum_{\eta} \biggl(\sum_{\substack{\gamma^+=\alpha,\\ \gamma\le\eta}} a_{\gamma\lambda}^{\eta}(t)\biggr)\E_{\eta}(t).
		\end{align*}
		On the other hand, since $\CMcal{L}_{\alpha}(t)$ is quasisymmetric \cite{HLMvW11} and $P_{\lambda}(t)$ is symmetric, their
		product is quasisymmetric and has a unique expansion in the $\CMcal{L}$-basis. Again by (\ref{E:defL}), we obtain
		\begin{align*}
			\CMcal{L}_{\alpha}(t)P_{\lambda}(t)&=\sum_{\beta}b_{\alpha\lambda}^{\beta}(t)\CMcal{L}_{\beta}(t)
			=\sum_{\beta}b_{\alpha\lambda}^{\beta}(t)  \sum_{\eta^+=\beta}\E_{\eta}(t)
			=\sum_{\eta}    b_{\alpha\lambda}^{\eta^+}(t)  \E_{\eta}(t).
		\end{align*}
		Since the set $\{\E_{\eta}(t)\,\vert\,\eta\in\mathbb{N}^n\}$ is a basis of $\mathbb{Q}(t)[x_1,x_2,\ldots,x_n]$, comparing coefficients in these two expansions gives (\ref{E:b2a}). Moreover, since $\L_{\alpha}(0)=\S_{\alpha}$ and $P_{\lambda}(0)=s_{\lambda}$, together with (\ref{E:QSxS}), we conclude $b_{\alpha\lambda}^{\beta}(0)=\langle \S_{\beta}\rangle \S_{\alpha}s_{\lambda}$.
		
		\qed
		
		{\em Proof of (\ref{eq:keycoeff}) of Corollary \ref{cor21}}.

		Recall that $w_0=n\,n-1\,\ldots\,1$ denotes the longest permutation of $[n]$. Since
		$\ell(w_0w)+\ell(w^{-1})=\ell(w^r)+\ell(w)=\ell(w_0)=\ell(w_0w w^{-1})$, we find $\Pi_{w_0}=\Pi_{w_0w}\Pi_{w^{-1}}$. On the other hand, the quadratic Hecke relation (defined in subsection \ref{ss:2.1}) implies $\Pi_i=tT_i^{-1}$, which implies
		\begin{align*}
			\Pi_{w^{-1}}=\Pi_{i_{l}}\cdots \Pi_{i_1}=t^l T_{i_{l}}^{-1}\cdots T_{i_{1}}^{-1}=t^l T_{w}^{-1}
		\end{align*}
		for any reduced expression $w=s_{i_1}\cdots s_{i_{l}}$ with $l=\ell(w)$. Equivalently,  $\Pi_{w^{-1}}T_w=t^{\ell(w)}$. 
		{Let $\mu=\lambda(\gamma)$, and let $w$ be the unique permutation of minimal length satisfying $w\mu=\gamma$.
			Then $(w_0ww_0)\mu^r=(w_0w)\mu=\gamma^r$ and $\ell(w)=\inv(\gamma^r)$.}
		Consequently, it follows from Definition \ref{Def:F} that
		\begin{equation}\label{eq:transfer}
			\Pi_{w_0}\E_\gamma(t)
			=\Pi_{w_0w}\Pi_{w^{-1}}T_wx^\mu
			=t^{\inv(\gamma^r)}\Pi_{w_0w}x^\mu
			=t^{\inv(\gamma^r)}K_{\gamma^r}(t).
		\end{equation}
		For every polynomial $f$ such that $s_i f=f$, and for every polynomial $g$, we have $\Pi_i(fg)=f\Pi_i(g)$ by definition. Thus, multiplication by $P_\lambda$ commutes with $\Pi_{w_0}$. Combining \eqref{eq:transfer} and Theorem~\ref{T:main} yields
		\begin{align*}
			t^{\inv(\gamma)}K_\gamma(t) P_\lambda(t)
			=\Pi_{w_0}\bigl(\E_{\gamma^r}(t)P_\lambda(t)\bigr)
			=\sum_{\eta\ge \gamma} a_{\gamma^r\lambda}^{\eta^r}(t)\,
			\Pi_{w_0}\E_{\eta^r}(t)
			=\sum_{\eta\ge \gamma} t^{\inv(\eta)}
			a_{\gamma^r\lambda}^{\eta^r}(t)K_\eta(t),
		\end{align*}
		as desired.
		\qed

		It remains to verify that specializing Corollary~\ref{cor21} at $t=0$ gives (\ref{E:sk1}). For this purpose, we introduce the Littlewood--Richardson skew key (LRK) tableaux from \cite[Section 6]{HLMvW10}. In what follows, we use an {\em \color{blue} increasing basement}: Basement $*$-symbols in the same row remain equal and $\ast$ is larger than every positive integer, while those in a higher row
		are considered smaller than those in a lower row.
		\begin{definition}[{\em \color{blue}Littlewood--Richardson skew key (LRK) tableaux}]Under the convention of an increasing basement, a skyline filling $\sigma$ is called a {\color{blue} Littlewood--Richardson skew key (LRK) tableau}
			if $\operatorname{col}(\sigma)$ is a regular contre-lattice word and every type $A$ and type $B$ triple is an inversion triple.
			
		\end{definition}
		
		{\em Proof of (\ref{E:sk1}) by (\ref{eq:keycoeff})}. We prove it by the same bijection $\psi$ as defined in \eqref{eq:set}, using the increasing large basement
		specified above.
		\begin{align*}
			\psi:\quad
			&\bigl\{
			T\in\SYT(\lambda';\eta,\gamma)\,\vert\,
			\inv^*(T)+h-m+\inv(\eta^r)-\inv(\gamma^r)=0
			\bigr\}\\
			&\quad
			\rightarrow
			\bigl\{
			\text{LRK tableaux of shape $\dg(\eta/\gamma)$
				with filling type $\lambda^r$}
			\bigr\},
		\end{align*}
		where $h=h(\gamma,\eta)$ and $m=m(\gamma,\eta)$. Write $\sigma=\psi(T)$. It has already been shown that $\sigma$ has filling type $\lambda^r$,
		that $\operatorname{col}(\sigma)$ is a regular contre-lattice word,
		and that its entries weakly decrease along each row. We only need to prove that
		$\inv^*(T)+h-m+\inv(\eta^r)-\inv(\gamma^r)=0$ forces
		every type $A$ and type $B$ triple in $\sigma$ to be an inversion triple.
		
		Fix $i<j$. Let $\coinv_{ij}(\sigma)$ denote the number of
		coinversion triples involving rows $i$ and $j$ of $\sigma$. Set $r=\min(\eta_i,\eta_j)$ and write
		$p_k=\sigma(i,k)$ and $q_k=\sigma(j,k)$.
		Define
		\[
		h_{ij}(\sigma)
		=
		\#\left\{
		k\,\vert\,
		\gamma_i<k\leq\eta_i,\,
		\gamma_j\leq k\leq\eta_j,\,
		q_{k+1}<p_k<q_k
		\right\},
		\]
		with the convention $q_{\eta_j+1}=\sigma(j,\eta_j+1)=0$.
		By \eqref{eq:zero-stat}, we have $\sum_{i<j}h_{ij}(\sigma)=\inv^*(T)+h-m$. We claim that the proof is complete once we show that
		\begin{equation}\label{eq:final}
			\coinv_{ij}(\sigma)
			=
			h_{ij}(\sigma)
			+\chi(\eta_i<\eta_j)-\chi(\gamma_i<\gamma_j).
		\end{equation}
		Indeed, summing \eqref{eq:final} over all $i<j$ gives
		\[
		\sum_{i<j}\coinv_{ij}(\sigma)
		=
		\inv^*(T)+h-m+\inv(\eta^r)-\inv(\gamma^r).
		\]
		Thus, for every $T\in\SYT(\lambda';\eta,\gamma)$, the expression
		on the right is the total number of coinversion triples in $\sigma$.
		In particular, it is nonnegative, and its vanishing implies that
		every type $A$ and type $B$ triple is an inversion triple.
		
		To prove (\ref{eq:final}), we assert that
		\begin{equation}\label{eq:claim}
			\sum_{0\le k\le r}\chi(q_{k+1}<p_k<q_k)
			=
			h_{ij}(\sigma)+\chi(\gamma_i\geq\gamma_j).
		\end{equation}
		We distinguish the cases where $p_k=*$ and where $p_k\ne *$. If $p_k\ne *$, then the contribution to the left-hand side of (\ref{eq:claim}) is precisely $h_{ij}(\sigma)$. Otherwise $p_k=\ast$; then necessarily $q_k=\ast$ and $q_{k+1}$ is an ordinary entry or the boundary value $0$.
		This occurs only when $k=\gamma_j\leq\gamma_i$,
		giving the additional contribution $\chi(\gamma_i\geq\gamma_j)$.
		
		We now  prove \eqref{eq:final}. First suppose that $\eta_i<\eta_j$. The triples involving rows $i$ and $j$ are of type $B$.
		By \eqref{eq:claim},
		\begin{align*}
			\coinv_{ij}(\sigma)=\sum_{k=0}^{r} \chi(q_{k+1}< p_k < q_k  )
			&=h_{ij}(\sigma) +\chi(\gamma_i\ge\gamma_j)
			=h_{ij}(\sigma) +1-\chi(\gamma_i<\gamma_j).
		\end{align*}
		This agrees with \eqref{eq:final} when $\chi(\eta_i<\eta_j)=1$. Now suppose that $\eta_i\geq\eta_j$, so that $r=\eta_j$. The triples involving rows $i$ and $j$ are of type $A$.
		By the non-attacking condition and the basement convention, we have $q_k\neq p_{k-1}$ for $1\leq k\leq r$.
		Using $p_k\leq p_{k-1}$ and $q_k\leq q_{k-1}$, and considering separately the cases $q_k<p_{k-1}$
		and $q_k>p_{k-1}$, we find
		\begin{align}\label{E:diff}
			\chi(p_k<q_k<p_{k-1})
			-\chi(q_k<p_{k-1}<q_{k-1})
			= \chi(p_k<q_k)-\chi(p_{k-1}<q_{k-1}).
		\end{align}
		Since the basement is strictly increasing from top to bottom, we have $p_0<q_0$. Moreover, $q_{r+1}=0<p_r$, so $\chi(q_{r+1}<p_r<q_r)=\chi(p_r<q_r)$. Therefore,
		\begin{align*}
			\coinv_{ij}(\sigma)&=\sum_{1\le k\le r}\chi(p_{k}< q_k < p_{k-1})\quad \mbox{ (by definition of a type $A$ triple)} \\
			&=\sum_{1\le k\le r}\chi(q_{k}< p_{k-1} < q_{k-1}  ) + \sum_{1\le k\le r}\left( \chi(p_{k}< q_k  )  -  \chi(p_{k-1}< q_{k-1} ) \right)\quad \mbox{ (by (\ref{E:diff}))}\\
			&= \sum_{0\le k<r}\chi(q_{k+1}< p_{k} < q_{k}  ) +\chi(p_r<q_r)-\chi(p_0<q_0) \\
			&=\sum_{0\le k<r}\chi(q_{k+1}< p_{k} < q_{k}  ) +\chi(q_{r+1}< p_r<q_r)-1 \quad \mbox{(by $p_0<q_0$ and $q_{r+1}<p_r$)}\\
			&=\sum_{0\le k\le r}\chi(q_{k+1}< p_{k} < q_{k}  )-1=h_{ij}(\sigma) +\chi(\gamma_i\ge\gamma_j)-1 \quad\mbox{ (by (\ref{eq:claim}))}.
		\end{align*}
		This proves \eqref{eq:final} when $\chi(\eta_i<\eta_j)=0$, thus completing the proof.
		
		\qed



		
		\section{Final remarks}\label{S:6}
		One would hope to introduce a parameter $q$ into Theorem \ref{T:main}, that is, to expand the product of a non-symmetric Macdonald polynomial and a symmetric Macdonald polynomial in terms of the non-symmetric ones and find a combinatorial formula for the coefficients of this expansion. Yip combined the combinatorics of alcove walks and double affine Hecke algebras to derive such a formula \cite[Theorem 4.2]{Yip}. However, it is unclear how that formula relates to Theorem \ref{T:main}. Moreover, with the additional parameter $q$, note that the coefficients are not in $\mathbb{N}[t,1-t,q,1-q]$ (that is, polynomials with nonnegative integer coefficients in $t$, $1-t$, $q$ and $1-q$).

		\section*{Acknowledgements}
		The second and third authors were supported by the National Natural Science Foundation of China (NSFC) under Project No. 12201529, and are currently supported by the NSFC under Project No. 12571358.

	\end{document}